\documentclass[smallcondensed]{svjour3}

\smartqed

\usepackage{graphicx}
\usepackage[utf8]{inputenc}
\usepackage{amsmath,amssymb}
\usepackage{framed}
\usepackage{tikz}
\usepackage{verbatim}
\usepackage{xcolor}
\usepackage{bbm}

\newtheorem{Def}{Definition}[section]
\newtheorem{Theorem}[Def]{Theorem}
\newtheorem{Lemma}[Def]{Lemma}

\makeatletter
\newcommand{\leqnomode}{\tagsleft@true\let\veqno\@@leqno}
\newcommand{\reqnomode}{\tagsleft@false\let\veqno\@@eqno}
\makeatother

\begin{document}

\title{Strong Convexity for Risk-Averse Two-Stage Models with Fixed Complete Linear Recourse
\thanks{The authors gratefully acknowledge the support of the German Research Foundation (DFG) within the collaborative research center TRR 154 ``Mathematical Modeling, Simulation and Optimization Using the Example of Gas Networks''.}
}

\titlerunning{Strong Convexity for Risk-Averse Linear Two-Stage Models}

\author{Matthias Claus \and Kai Sp\"urkel}

\authorrunning{M. Claus and K. Sp\"urkel}

\institute{M. Claus \at University Duisburg-Essen \\ Thea-Leymann-Straße 9 \\ D-45127 Essen \\ Tel.: +49 201 183 6887 \\ \email{matthias.claus@uni-due.de}
\and
K. Sp\"urkel \at University Duisburg-Essen \\ Thea-Leymann-Straße 9 \\ D-45127 Essen \\ Tel.: +49 201 183 6890 \\ \email{kai.spuerkel@uni-due.de}}

\date{Received: date / Accepted: date}

\maketitle

\begin{abstract}
This paper generalizes results concerning strong convexity of two-stage mean-risk models with linear recourse to distortion risk measures. Introducing the concept of (restricted) partial strong convexity, we conduct an in-depth analysis of the expected excess functional with respect to the decision variable and the threshold parameter. These results allow to derive sufficient conditions for strong convexity of models building on the conditional value-at-risk due to its variational representation. 
Via Kusuoka representation these carry over to comonotonic and distortion risk measures, where we obtain verifiable conditions in terms of the distortion function. 
For stochastic optimisation models, we point out implications for quantitative stability with respect to perturbations of the underlying probability measure.
Recent work in \cite{Ba14} and \cite{WaXi17} also gives testimony to the importance of strong
convexity for the convergence rates of modern stochastic subgradient descent algorithms and 
in the setting of machine learning.


\end{abstract}

\section{Introduction}
In \cite{S94} Schultz considered the expectation based 
two-stage optimisation problem
\begin{align}\label{OptProblem1}
\min\{ \ \mathbb{E}_\omega[ \, h(\xi) +  \varphi(z(\omega) - T\xi) \, ] \ : \ \xi \in X \}
\end{align}
where $X$ is a subset of some $\mathbb{R}^n$, $h$ convex and real-valued,
$z(\omega)$ a $s$-dimensional random vector on some propability space $(\Omega,\mathcal{F}, \mathbb{P})$ and 
$\varphi$ a recourse function of the form
\begin{align}\label{valuefunction}
\varphi(t) = \min\{ q^\intercal y \ | \ Wy = t, y \geq 0 \}
\end{align}
which is the value function of a linear program with parametric right-hand side.
For a general introduction to such models we refer to the standard
textbooks \cite{BL11} and \cite{SDR14}.
The optimisation problem \eqref{OptProblem1} can be rewritten as
\begin{align}\label{OptProblem2}
\min \{\hat{h}(x) + Q_\mathbb{E}(x) \ | \ x \in T(X) \}
\end{align}
with
\begin{align}
\label{Hath}
\hat{h}(x) = \min\{ h(\xi) \ | \ T \xi = x, \ \xi \in X \}
\end{align}
and the reduced expectation function
\begin{align}\label{QExpectation}
Q_\mathbb{E}(x) = \mathbb{E}_\mu[ \, \varphi(z-x)  \, ]
\end{align}
where $\mu$ denotes the pushforward measure of $z$. It is well-known that under mild assumptions $Q_\mathbb{E}$ is well-defined and convex on all of $\mathbb{R}^s$. 
For further structural analysis of the optimisation problem \eqref{OptProblem1}
(e.g. stability analysis, cf. section 4) and its algorithmic treatment with
subgradient schemes (cf. \cite{Nes04}) conditions for strong convexity may be desirable  
and for the risk-neutral setting were given in \cite{S94}, Theorem 2.2:
\begin{theorem}\label{TheoremSCExp} 
Assume that the following conditions are satisfied:
\begin{itemize}
\item[A1] For every $t$ there exist some $y \geq 0$ such that $W y = t$. 
(Complete recourse)
\item[A2] There exists some $v$ with $W^\intercal v < q$. 
(Strengthened sufficiently expensive recourse)
\item[A3] $\|z\|$ is $\mu$-integrable. (Finite first moments)
\item[A4] $\mu$ has a density $\theta$ with respect to the Lebesgue-measure and 
there exists a convex open set $V$ , constants $r, \rho > 0$ such that 
$\theta \geq  r$ a.s. on $V + B_\rho(0)$.
\end{itemize}
Then $Q_\mathbb{E}$ is strongly convex on $V$.
\end{theorem}
Remember that a real-valued function $f$ on some convex subset $V$ of
a normed space is called $\kappa$-strongly convex on that set if for all 
$x,y \in V$ and all 
$\lambda \in \left( 0,1 \right)$ it holds
\begin{align*}
f(\lambda x + (1-\lambda) y ) \leq \lambda f(x) + (1-\lambda) f(y) - 
\frac{\kappa}{2}\lambda(1-\lambda)\|x-y\|^2.
\end{align*}
We point out that a constant of strong convexity for $Q_\mathbb{E}$ in Theorem \ref{TheoremSCExp}
can be computed from the model data, i.e. the geometry
of the set $\{ v \ | W^\intercal v \leq q\}$, $\rho$ and $r$.\\
In \cite{CSS17} the analysis of \eqref{QExpectation} was extended to the upper semideviation based functional 
\begin{align*}
Q_{\mathcal{D}_+}(x)  
= Q_\mathbb{E}(x) + \mathbb{E}_\mu[ \max\{ 0, \varphi(z-x) - Q_\mathbb{E}(x) \} ]
\end{align*}
and the expected-excess based one
\begin{align*}
Q_{EE}(x,\eta)  
= \mathbb{E}_\mu [ \max\{ 0, \varphi(z-x) - \eta \} ]
\end{align*}
which, for simplicity,  we shall call upper semideviation and expected excess repectively.
For the latter one an additional assumption A5 on the magnitude of $\eta$ 
is needed because if $\eta$ is too big, $Q_{EE}(x,\eta)$ might not even depend on $x$ anymore.

\subsection{On strong convexity of $Q_{EE}(\cdot,\eta)$ for fixed $\eta$}

In order to formulate condition A5 we note the following properties of the value function $\varphi$ and its linearity complex (cf. Lemma 32 and 34 in \cite{CSS17}): 
\begin{lemma} \label{lemma0}
Assume A1 and A2. Then $\{ v \ | \ W^\intercal v \leq q \}$ is the convex hull of its finitely many extreme points
$\{ d_i \ | \ i \in I = \{ 1,\ldots,N \} \}$ and $\varphi$ has the following properties:
\begin{itemize}
\item[(i)] $\varphi(t) = \max_{i \in I} d_i^\intercal t$,  $\varphi(t) = d_i^\intercal t$  for all $t \in K_i = \{ z \ | \ (d_i - d_j)^\intercal z \geq 0 \text{ for all } j \in I  \}$, i.e. $\varphi$ is finite and polyhedral.
\item[(ii)] $\bigcup_{i \in I} K_i = \mathbb{R}^s$ with each $K_i$ being a $s$-dimensional, pointed polyhdral cone. 
Furthermore each $K_i \cap K_j$ with $i \neq j$ is a common closed face of $K_i$ and $K_j$ and it holds
$\text{dim}(K_i \cap K_j) = s-1$ if and only if $d_i$ and $d_j$ are adjacent. 
\item[(iii)] There is some $\alpha > 0$ such that 
\begin{align*}
\inf_{u \in K_i} \sup_j (d_i - d_j)^\intercal u \geq \alpha \|u\|.
\end{align*}
\end{itemize}
\end{lemma}
Let us fix some more notation:\\
Since each $K_i$ is a polyhedral cone, we can write it as the conic hull of its finitely many extreme-rays, i.e. 
$K_i = \text{cone}\{ t^i_1,\ldots,t^i_{n_i} \}$. With shorthand $K_i^+ = K_i \cap \{ z \ | \ d_i^\intercal z \geq 0 \}$
and $I^+ = \{ i \in I \ | \ \text{int}K_i^+ \neq \emptyset \text{ and } d_i \neq 0\}$ we note that for $\eta_0 > 0$ and $i \in I^+$ it holds
that the hyperplane $\{ z \ | \  d_i^\intercal z = \eta_0 \}$ intersects at least one extreme ray of $K_i$ in a single point:
\begin{align*}
\{ r t^i_{s} \ | \ r \geq 0 \} \cap \{ z \ | \ d_i^\intercal z = \eta \} = \{ \hat{y}^i_s \}
\end{align*}
for at least one $s \in \{ 1,\ldots,n_i \}$. Let $\hat{y}^i(\eta_0)$ denote one with minimum norm.
Theorem 35 in \cite{CSS17} can then be formulated as this:
\begin{theorem}\label{theorem0}
Let A1-A4 hold. In addition assume
\begin{itemize}
\item[A5] $\eta_0$ is such that for all $i \in I^+$ we have $\|\hat{y}^i(\eta_0)\| 
< \rho$ (where $\rho$ is the one given in A4).
\end{itemize}
Then $Q_{EE}(x,\eta)$ is strongly convex on $V$ (cf. A4 for the definition of $V$) with respect to 
$x$ for all $\eta \leq \eta_0$. The modulus of strong convexity does not depend on $\eta$.
\end{theorem}
The geometric situation is shown in Fig. 1.\\
In A5 it is in fact enough to show that for every $i \in I^+$ it holds 
$\|\hat{y}^i(\eta_0)\| < \rho$ or if there exist an index set $J_i \subset I$ such that $-K_i \subset \bigcup_{j \in J_i} K_j^+$
it holds $\|\hat{y}^j(\eta_0)\| < \rho$ for all $j \in J_i$. In this paper we shall 
use the slightly less general version of A5.
\begin{figure}
  \caption{The geometry of the cones $K_i$}
  \centering
    \includegraphics[width=1.0\textwidth, trim={0 6cm 0 6cm},clip]{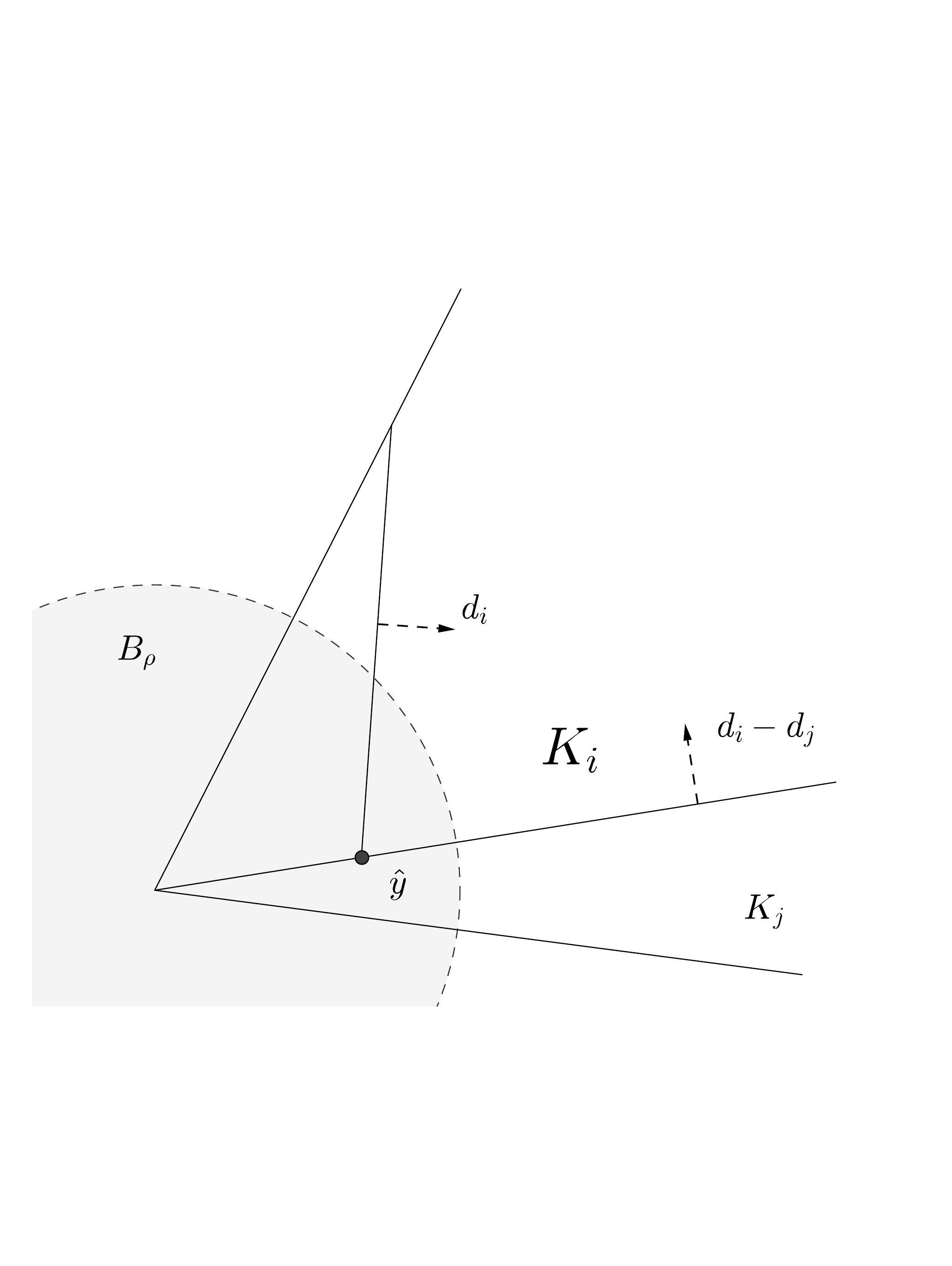}
\end{figure}
\newpage

Let us make three remarks on the theorem:\\
Firstly, it is desirable to verify condition A5 for $\eta_0$ as large as possible
(especially when considering Theorem \ref{theorem2} later).
The larger $\rho$ and $\|d_i\|$ with $i \in I^+$, the larger $\eta_0$ can be chosen.\\
Secondly, condition A5 might not be fulfilled on the entire set $V$.
By considering a subset $U \subset V$ instead, allowing for a possibly bigger value of
$\rho$, one can ensure that condition A5 holds on $U$ at least.
Hence, strong  convexity of $Q_{EE}(\cdot, \eta)$ can be shown on a smaller set. 
We shall demonstrate these two remarks in Example 1 below. \\ 
Thirdly, a modulus of strong convexity $\kappa$ can be computed in terms of model-data directly
with the techniques employed in \cite{CSS17}. It depends on $\rho$, the shape of 
$\{ W^\intercal v \leq q\}$, the lower bound of $\mu$'s density $r$ and $\eta_0$
(all subsumized as ''model data''). There is a certain trade-off between how big $\kappa$
can be and how $\eta_0$ is chosen: The bigger $\eta_0$, the smaller $\kappa$.
Example 2 illustrates this fact.\\
In Theorem \ref{theorem2} and Theorem \ref{theorem3} we will calculate moduli of strong convexity for $Q_{EE}$ in a more general setting than in Theorem \ref{theorem0}.

\begin{example}
Consider $\varphi(t) = \max\{ 0, \alpha \, t \}$ and $\mu =\mathbb{U}_{| (-\rho,1+\rho)}$.
For arbitrary $x \in V = (0,1)$ we compute for $\eta, \alpha, \rho > 0$
\begin{align*}
Q_{EE}(x) = 
\begin{cases}
0 \text{ for } x \geq 1 + \rho-\alpha^{-1}\eta \\
\frac{1}{2 \, \alpha} [ \alpha(1 + \rho -x)-\eta ]^2 , 
\text{ else.}
\end{cases}
\end{align*}
We see that $Q_{EE}(\cdot,\eta)$ is strongly convex as long as the condition
$0 < x < \min\{ 1+\rho-\alpha^{-1}\eta, 1 \}$ is fulfilled.
In notation of A5 there is only one cone $K_i^+$ with $i \in I^+$ which 
is $K = \mathbb{R}_{\geq 0}$ corresponding to $d = \alpha$.
We get $\hat{y} = \alpha^{-1}\eta_0$ and condition A5 reads $\alpha^{-1}\eta_0 < \rho$ so
that $Q_{EE}(\cdot,\eta)$ is strongly convex on $V$ whenever $\eta < \alpha \rho$.
If $\rho < \frac{1}{2}$ we may consider the set $U = (\rho, 1-\rho) \subset V$ to get
strong convexity for larger $\eta$ than before. We can set as the new 
$\tilde{\rho} = 2\rho$ and see that $Q_{EE}(\cdot,\eta)$ is strongly
convex on $U$ for all $\eta < 2 \alpha \rho$. 
\end{example}

Note that the modulus of strong convexity does not depend 
on $\rho$ or $\eta$  which in higher dimensional cases cannot be expected:

\begin{example}
Let $\varphi(t) = \max\{ 0,t_1,t_2 \}$, $\mu = \mathbb{U}_{ \left(0,1\right) }$.\\
We shall also assume $0< x_2 < x_1$, $\eta > 0$ and $x_1 + \eta \leq 1$ when computing
\begin{align*}
Q_{EE}(x_1,x_2) &= \frac{1}{2}x_2(1-x_1-\eta)^2 + \frac{1}{2}x_1(1-x_2-\eta)^2 
+ \frac{2}{3}(1-x_1-\eta)^3 \\ 
&+ \eta (1-x_1-\eta)^2 + \frac{1}{2}(1-x_1) [ (1-x_2-\eta)^2 - (1-x_1-\eta)^2 ].
\end{align*}
For the tedious computation we refer to the appendix.
The components of the Hessian $\mathcal{H}$ of $Q_{EE}$ are
\begin{align*}
\frac{\partial^2}{\partial x_1^2}Q_{EE}(x_1,x_2) &= x_2 - x_1 + 1 \\
\frac{\partial^2}{\partial x_2^2}Q_{EE}(x_1,x_2) &= 1 \\
\frac{\partial^2}{\partial x_1 \partial x_2}Q_{EE}(x_1,x_2) &= x_1 + \eta - 1.
\end{align*}
If $\eta \rightarrow 1-$, we can choose $x_1$ close to $1$ and $x_2$ close to $0$ which gives
$\text{det}\mathcal{H}_{Q_{EE}} \rightarrow 0+$. Since the determinant of the Hessian
is equal to the product of its Eigenvalues we see that at least one of them approaches $0$ so 
that the modulus of strong convexity must depend on the choice of $\eta_0$ given in condition A5.
\end{example}

\subsection{Variational representation of CVaR}
In section 3 we shall consider the Conditional Value-at-Risk at a confidence level $\alpha \in (0,1)$, 
which can be characterized by a minimisation problem in terms of $Q_{EE}$:
\begin{equation}\label{varrep}
Q_{\alpha CVaR}(x) = \min_{\eta \in \mathbb{R}} \big\{  \eta + \frac{1}{\alpha} Q_{EE}(x,\eta) \big\}.
\end{equation}
For smooth distributions this representation is due to \cite[Theorem 1]{UR99}, while \cite[Theorem 10]{UR02} covers the general case and shows that the Value-at-Risk
$$
Q_{\alpha VaR}(x) = \inf \Big\{ t \in \mathbb{R} \; | \; \mu(\lbrace z \; | \; \varphi(z-x) \leq t \rbrace) \geq 1-\alpha \Big\}
$$
is an optimal solution to the above minimisation problem. Thus,
\begin{align}\label{varrep2}
Q_{\alpha CVaR}(x) = Q_{\alpha VaR}(x) + 
\frac{1}{\alpha} Q_{EE}(x,Q_{\alpha VaR}(x)).
\end{align}
We shall however favor working with representation \eqref{varrep} due to inconvenient properties of $Q_{\alpha VaR}$, 
i.e. absence of convexity. 
While it is straightforward to show convexity of $Q_{\alpha CVaR}$ through \eqref{varrep} essentially due to
joint convexity of $Q_{EE}$ in both arguments, strong convexity
does not follow trivially even if $Q_{EE}(x,\eta)$ is strongly convex in $x$ with strong convexity constant
not depending on $\eta$. A property that ensures strong convexity of $Q_{\alpha CVaR}$ can be defined as follows:

\begin{definition}\label{def1}
Let $V \subset \mathbb{R}^n$ and $W \subset \mathbb{R}^m$ nonempty and convex.\\
A function $f: V \times W \rightarrow \mathbb{R}$ is called partially 
$\kappa$-strongly convex with respect to its first argument if
\begin{align}\label{PSC}
f( \lambda(x_1,y_1) + (1-\lambda)(x_2,y_2 ))
\leq \lambda f(x_1,y_1) + (1-\lambda) f(x_2,y_2) - \frac{\kappa}{2}\lambda
(1-\lambda)\|x_1 - x_2\|^2
\end{align}
holds for all $x_1,x_2 \in V, \, y_1,y_2 \in W, \, 0 < \lambda < 1$.
\end{definition}
\begin{lemma}\label{lemma2}
If $f$ (as above) is continuously differentiable then partial strong convexity is equivalent to
\begin{align}\label{MonMap}
[ f'(x_1,y_1) - f'(x_2,y_2) ]((x_1,y_1)-(x_2,y_2)) \geq 
\kappa \|x_1 - x_2\|^2
\end{align}
for all $x_1,x_2 \in V, \, y_1,y_2 \in W, \, 0 < \lambda < 1$.
\end{lemma}
Although the proof is virtually the same as for strong convexity (cf. \cite{RhOr70})
in both arguments, we shall give a variant of the proof in for the reader's convenience in the appendix.\\

For the moment let us assume that $Q_{EE}(x, \eta)$ is partially strongly convex with respect to $x$ 
on some set $V \times W \subset \mathbb{R}^s \times \mathbb{R}$.
The following simple calculation shows that $Q_{\alpha CVaR}$  is strongly convex on $V$ with modulus $\frac{\kappa}{\alpha}$ if $Q_{\alpha VaR}(V) \subset W$.

For any $\lambda \in [0,1]$ and $x_1, x_2 \in V$ set 
$\eta_i = Q_{\alpha VaR}(x_i) \in W$, $i=1,2$. Then
\begin{align*}
&Q_{\alpha CVaR}(\lambda x_1 + (1-\lambda)x_2) \; = \; 
\min_{\eta \in \mathbb{R}} \; \big\{ \eta + 
\frac{1}{\alpha} Q_{EE}(\lambda x_1 + (1-\lambda)x_2, \eta) \big\} \\
\leq \; & \lambda \eta_1 + (1-\lambda)\eta_2 + 
\frac{1}{\alpha} 
Q_{EE}(\lambda x_1 + (1-\lambda)x_2, \lambda \eta_1 + (1-\lambda)\eta_2) \\
\leq \; &\lambda \, [\eta_1  + \frac{1}{\alpha}Q_{EE}(x_1, \eta_1)] + 
(1-\lambda) \, [ \eta_2  + \frac{1}{\alpha} Q_{EE}(x_2,\eta_2) ] - \frac{\kappa}{\alpha} \lambda (1-\lambda) \|x_1 - x_2 \|^2 \\
= \; &\lambda \, Q_{\alpha CVaR}(x_1) + (1-\lambda) \, Q_{\alpha CVaR}(x_2) - \frac{\kappa}{\alpha} \lambda (1-\lambda) \|x_1 - x_2 \|^2,
\end{align*}
where the second inequality follows from $Q_{EE}$ being partially strongly convex.\\ 

One might hope that conditions A1-A5 suffice to prove partial strong convexity for $Q_{EE}$. It turns out, that
this is not true:

\begin{example}
As a counterexample consider
\begin{align*}
Q_{EE}(x,\eta) = \int_\mathbb{R} \max\{0 , \varphi(z-x)-\eta \} \, \mu(\text{d}z)
\end{align*}
with $\varphi(t) = \max\{0,t\}$ and $\mu = \mathbb{U}_{| \, \left(0,1\right)}$ (i.e. uniformly distributed on 
$\left( 0,1 \right))$. Let $V \subset \left( 0,1 \right)$ so that conditions $A1-A4$ are satisfied.
Also choose $\eta_0 > 0$ so that A5 holds.
In case $\eta_0 \geq \eta > 0$ we get
\begin{align*}
Q_{EE}(x,\eta) &= \int_x^1 \max\{ 0, z-x-\eta \} \, \text{d}z 
= \int_{x+\eta}^1 z-x-\eta \, \text{d}z \\
&= \int_0^{1-x-\eta} t \, \text{d}t = \frac{1}{2}(1-x-\eta)^2.
\end{align*}
We calculate for $x, x+u \in V$
\begin{align*}
[Q_{EE}'(x+u,\eta+\nu) - Q_{EE}'(x,\eta)](u,\nu) = u^2 + 2u\nu + \nu^2 = (u + \nu)^2
\end{align*}
as long as $\eta+\nu \geq 0$. Since for small $\nu$ we can choose $u = -\nu$,
$Q_{EE}$ cannot be partially strongly convex in wrt. $x$ on the set 
$V \times \left( -\infty, \eta_0 \right]$. \\ 
In this example $Q_{EE}$ does not have the desired joint convexity properties on any open set contained 
in the support of $\mu$.
\end{example}

Unsurprisingly, $Q_{\alpha CVaR}$ does not behave any better:

\begin{example}
With the same specifications as in example 3 we can calculate the conditional value at risk
at some level $\alpha \in \left( 0,1 \right)$ for any $x \in V$ as (cf. \eqref{varrep})
\begin{align*}
Q_{\alpha CVaR}(x) = \min_{\eta \in \mathbb{R}} \big\{  \eta + \frac{1}{2\alpha} (1-x-\eta)^2  \big\}.
\end{align*}
The unique minimizer is $\eta^* = Q_{\alpha VaR}(x) = 1-\alpha - x$. 
Note that we can restrict the minimisation to $\eta > 0$ since 
the value at risk is obviously positive. We arrive at
\begin{align*}
Q_{\alpha CVaR}(x) = -x + \frac{1}{2}(2-\alpha),
\end{align*}
so there is no subset $U \subset V$ on which $Q_{\alpha CVaR}$ is strongly convex.
\end{example}

We shall give a rather strict condition on the value function $\varphi$ that yields partial strong convexity 
for the expected excess in section 3 without additional assumptions on the distribution of $z$.

Before that we will introduce the even weaker concept of strong convexity, restricted partial strong convexity,
which can be shown to hold for $Q_{EE}$ under less restrictive assumptions on the recourse function.
This property can also be characterized by monotonicity of the gradient as in Lemma 
\ref{lemma2}. 
\begin{definition}\label{def2}
Let $V \subset \mathbb{R}^n$ and $W \subset \mathbb{R}^m$ nonempty and convex.\\
A function $f: V \times W \rightarrow \mathbb{R}$ is called restricted partially 
$\kappa$-strongly convex on $\Omega \subset V \times W$ with respect to its first
argument if
\begin{align*}
f( \lambda(x_1,y_1) + (1-\lambda)(x_2,y_2) ) 
\leq \lambda f(x_1,y_1) + (1-\lambda) f(x_2,y_2) - 
\frac{\kappa}{2}\lambda(1-\lambda)\|x_1-x_2\|^2
\end{align*}
for all $(x_1,y_1),(x_2,y_2) \in \Omega$ and all $0 < \lambda < 1$.
\end{definition}
Note that $\Omega$ does not need to be a cylindrical subset of $V \times W$ as in 
Definition \ref{def1}.\\

In section 3 we will show that conditions A1-A5 are sufficient for restricted partial strong convexity 
of $Q_{EE}$ on some nonempt set $\Omega \subset V \times \mathbb{R}$.  
The estimates above showing that partial strong convexity of $Q_{EE}$ implies strong convexity 
of $Q_{\alpha CVaR}$ can be used verbatim to show that restricted partial strong convexity of $Q_{EE}$
does so as well - if one can additionally show that $(x, Q_{\alpha VaR}(x)) \in \Omega$ for all $x \in V$ (or maybe
for all $x \in U \subset V$). Example 4 shows that this cannot be done without only relying on assumptions A3 and
A4 on the distribution of $z$.

\section{Joint properties of $Q_{EE}$}
Theorem \ref{theorem0} addresses properties of $Q_{EE}(\cdot,\eta)$ for fixed threshold $\eta$. 
In Theorem \ref{theorem2} we will show that in conjunction with A1-A5 the following condition is
sufficient for partial strong convexity of $Q_{EE}(x,\eta)$ wrt. $x$:
\begin{itemize}
\item[A6] It holds $q > 0$, i.e. the gradient of the objective function 
of the second stage is positive componentwise  (cf. \eqref{valuefunction}).
\end{itemize}
Note that this condition is stronger than A2 (choose $v = 0$ there).
Although it is a rather strict condition on the problem data, it might be well 
justifiable in the setting of simple recourse problems because compensating actions 
in the second stage should have negative impact on the total objective.\\
If A6 does not hold, all that can be shown is restricted partial strong convexity in the sense
of Definition \ref{def2}. We start with an elementary lemma to provide some geometrical insights
that are used within the proof of Theorems \ref{theorem2} and \ref{theorem3}:

\begin{Lemma}\label{lemma1}
Let A1 and A2 hold. Then A6 is fulfilled if and only if one of the following conditions is fulfilled:
\begin{itemize}
\item[(i)] There is some $\alpha' > 0$ such that for all 
$u \in \mathbb{R}^s$ it holds
\begin{align}\label{est2}
\varphi(u) \geq \alpha' \, \|u\|,
\end{align}
\item[(ii)] For all $i \in I$ we have $d_i \in \text{int cone}\{ d_i - d_j \ | \ d_j \text{ adj. to } d_i \}$.
\end{itemize}
\end{Lemma}
\begin{proof}
This follows directly from well-known separating hyperplane theorems.
\end{proof}

\begin{Lemma}
Let A1-A4 hold. Then $Q_{EE}$ is continuously differentiable and we have the following formula:
\begin{align}\label{gradient_formula}
[Q_{EE}'(x+u, \eta + \nu) - Q_{EE}'(x,\eta)](u, \nu) \notag \\
= \int \mu\big( \bigcup_{l \in I(u,\nu)(\tau)} M_l(x, \eta) \big\backslash
\bigcup_{l \in I(u,\nu)(\tau)} M_l(x+u, \eta+\nu) \big) \, \text{d}\tau
\end{align}
with suitable parametric sets $I(u,\nu)$ sets $M_l(x,\eta) \subset \mathbb{R}^s$ to be constructed in the 
proof below.
\end{Lemma}
\begin{proof}
By assumptions A1-A4 and standard arguments $Q_{EE}$ is continuously differentiable on 
$\mathbb{R}^s \times \mathbb{R}$, so only \eqref{gradient_formula} warrants a proof.
We calculate 
\begin{align}\label{EEderiv}
\nonumber
&Q_{EE}'(x;\eta)(u,\nu) = 
\sum_{i \in I} \mu(M_i(x,\eta))(-d_i^\intercal u - v) \\
&Q_{EE}'(x+u;\eta+\nu)(u,\nu) = 
\sum_{i \in I} \mu(M_i(x+u,\eta+\nu))(-d_i^\intercal u - v)  
\end{align}
with shorthand $M_i(x,\eta) = (x + K_i) \cap \{ \varphi(z-x) \geq \eta \}$.\\
Let $M_0(x,\eta) = \{ \varphi(z-x) \leq \eta \}$ and 
\begin{align*}
&\pi^1_i = \mu(M_i(x,\eta)), &y^1_i = -d_i^\intercal u - v, \, i \in I, \\
&\pi^1_0 = \mu(M_0(x,\eta)), &y^1_0 = 0, \\
&\pi^2_i = \mu(M_i(x+u,\eta+\nu)), &y^2_i = -d_i^\intercal u - v, \, i \in I, \\
&\pi^2_0 = \mu(M_0(x+u,\eta+\nu)), &y^2_0 = 0.
\end{align*}
Define random variables $Y_1$ and $Y_2$ taking values $y^1_i$ with probability
$\pi^1_i$ and $y^2_i$ with probability $\pi^2_i$ respectively.
We observe that the quantities in \eqref{EEderiv} can be rewritten as 
Riemann-Stieltjes integrals with cdfs $F_{Y_1}, F_{Y_2}$ as integrators: 
\begin{align*}
&Q_{EE}'(x,\eta)(u,\nu) = \sum_{i=0}^N \pi^1_i y^1_i = 
\int \tau \, \text{d}F_{Y_1}(\tau) \\
&Q_{EE}'(x+u,\eta+\nu)(u,\nu) = \sum_{i=0}^N \pi^2_i y^2_i = 
\int \tau \, \text{d}F_{Y_2}(\tau).
\end{align*}
Integration by parts yields
\begin{align*}
&[Q_{EE}'(x+u,\eta+\nu) - Q_{EE}'(x,\eta)](u,\nu) \\
&= \int \tau \, \text{d}F_{Y_2}(\tau) - \int \tau \, \text{d}F_{Y_1}(\tau)
= \int F_{Y_1} - F_{Y_2} \, \text{d}\tau.
\end{align*}
Note that the boundary terms cancel out because $y^1_i = y^2_i$ for all $i$.\\
Introducing the index set
\begin{align*}
I(u,\nu)(\tau) = \{i \in I \cup \{ 0 \} \ | \ y_i \leq \tau \}
\end{align*}
and observing that the sets $M_i(x,\eta)$ only meet in lower dimensional
sets (if they meet at all) and thus $\mu(M_i(x,\eta) \cap M_j(x,\eta)) = 0$ for $i \neq j$,
we can write down the cdfs $F_{Y_1}$ and $F_{Y_2}$ as follows:
\begin{align}\label{cdfs}
\nonumber
&F_{Y_1}(\tau) = \sum_{i \in I(u,\nu)(\tau)} \pi^1_i = 
\sum_{i \in I(u,\nu)(\tau)} \mu(M_i(x,\eta)) =
\mu\big( \bigcup_{i \in I(u,\nu)(\tau)} M_i(x,\eta) \big), \\
&F_{Y_2}(\tau) = \sum_{i \in I(u,\nu)(\tau)} \pi^2_i = 
\sum_{i \in I(u,\nu)(\tau)} \mu(M_i(x+u,\eta+\nu)) =
\mu\big( \bigcup_{i \in I(u,\nu)(\tau)} M_i(x+u,\eta+\nu) \big).
\end{align}
Note that $F_{Y_1} - F_{Y_2}$ can be written as the measure of a difference
of sets if we can show that the following inclusion holds for all $\tau \in \mathbb{R}$:
\begin{align}\label{inclusion1}
\bigcup_{i \in I(u,\nu)(\tau)} M_i(x+u, \eta+\nu) \subset 
\bigcup_{i \in I(u,\nu)(\tau)} M_i(x, \eta).
\end{align}
First note that for all $\tau$ we have
\begin{align}\label{inclusion2}
\bigcup_{i \in I(u,\nu)(\tau) \backslash \{0\}} x + u + K_i
\subset  \bigcup_{i \in I(u,\nu)(\tau) \backslash \{0\}} x + K_i
\end{align}
for if it was not true there would be some $\bar{z}$ contained in the union
to the left side of the inclusion but not in the right.
This means there is some index $i_1$ with $\bar{z} \in x+u+K_{i_1}$
such that $-d_{i_1}^\intercal u - v \leq \tau$.
Since the cones $x+K_i$ cover the entire space (cf. Lemma \ref{lemma0} (ii) ) there is some index
$i_2$ such that $\bar{z} \in x + K_{i_2}$ and $-d_{i_2}^\intercal u - v > \tau$.
By using the definition of the sets $K_i$ we arrive at the contradiction
\begin{align*}
0 &\leq (d_{i_1} - d_{i_2})^\intercal (\bar{z}-x-u)
= (d_{i_1} - d_{i_2})^\intercal (\bar{z}-x) - (d_{i_1} - d_{i_2})^\intercal u \\
&\leq -d_{i_1}^\intercal u + d_{i_2}^\intercal u 
= (-d_{i_1}^\intercal u - v) + (d_{i_2}^\intercal u + v) < \tau + (-\tau) = 0
\end{align*}
Back to \eqref{inclusion1} we see that this inclusion reduces to 
\eqref{inclusion2} in case $\tau \geq 0$ which was just discussed.
Now let $\tau < 0$ and $\bar{z} \in M_{i_1}(x+u,\eta+\nu)$ for some $i_1 \in I(u,\nu)(\tau)$
which implies  $\bar{z} \in x + u + K_{i_1}$,  
$-d_{i_1}^\intercal u - \nu \leq \tau < 0$ and 
$d_{i_1}^\intercal (\bar{z} - x - u) \geq \eta + \nu$.
This yields
\begin{align*}
\varphi(\bar{z}-x) \geq d_{i_1}^\intercal (\bar{z}-x-u) + d_{i_1}^\intercal u 
> \eta + \nu - \nu = \eta. 
\end{align*}
Since by \eqref{inclusion2} we also have $\bar{z} \in x + K_{i_2}$ for some
$i \in I(u,\nu)(\tau)$ we have shown that $\bar{z} \in \bigcup_{i \in I(u,\nu)(\tau)} M_i(x, \eta)$ and
\eqref{inclusion1} is proven. We can now replace $F_{Y_1} - F_{Y_2}$ and conclude the proof with
\begin{align*}
\int (F_{Y_1} - F_{Y_2})(\tau) \, \text{d}\tau 
= \int \mu\big( \bigcup_{l \in I(u,\nu)(\tau)} M_l(x, \eta) \big\backslash
\bigcup_{l \in I(u,\nu)(\tau)} M_l(x+u, \eta+\nu) \big) \, \text{d}\tau.
\end{align*}
\qed
\end{proof}

Since $Q_{EE}$ is continuously differentiable we can prove (restricted) partial strong 
convexity of $Q_{EE}$ by showing that \eqref{MonMap} (and its restricted counterpart) 
holds for $Q_{EE}$, i.e. showing that there exists some $\kappa > 0$ with
\begin{align}\label{SCformula}
[Q_{EE}'(x+u, \eta + \nu) - Q_{EE}'(x,\eta)](u,\nu) \geq \kappa \, \|u\|^2
\end{align}
for relevant $x,\eta,u, \nu$. \\
This is done by restricting the area of integration in \eqref{gradient_formula} 
to some subset with measure not smaller than $\alpha \|u\|$ for some constant $\alpha > 0$. 
Then the $\mu$-measure of the set within the integrand will be estimated from below by constructing 
a cylindrical subset with Lebesgue measure not smaller than $\alpha' \|u\|$ (with some other constant
$\alpha' > 0$) and which is contained in $V+B_\rho(0)$, where a lower bound on $\mu$'s 
Lebesgue-density is available. We begin with the special case employing condition A6:

\begin{Theorem}[Partial strong convexity of $Q_{EE}$]\label{theorem2}
Let A1-A6 hold.\\Then $Q_{EE}(x,\eta)$ is partially strongly convex on the set 
$V_{\eta_0} = V \times \left( -\infty, \eta_0 \right]$ wrt. $x$.
\end{Theorem}

\begin{proof}
Crucially relying on A6 we may assume $\nu \geq 0$,  otherwise substitute $x' = x + u, \eta' = \eta + \nu$ and 
$u' = -u, \nu' = -\nu$ and consider $u'$ and $\nu'$ instead of $u$ and $\nu$.  
Since the sets $K_l$ cover the entire space $\mathbb{R}^s$ we can pick an index
$i$ such that $u \in K_i$. Note that in condition A5 and the discussion before it, 
which will come into play soon, we have $I = I^+$ and $K_l^+ = K_l$ for all indices $l \in I$. \\

We shall now construct some $\eta^- > 0$ and give the desired estimate of \eqref{gradient_formula} 
for $-\infty < \eta \leq \eta^-$ first.
 By Lemma \ref{lemma0} (iii) there is some index $j$ different from $i$ such that 
\begin{align}\label{est1}
(d_i - d_j)^\intercal u \geq \alpha \|u\|.
\end{align}

Assume that $\tau \in ( -d_i^\intercal u - \nu,  -d_j^\intercal u - \nu )$. This implies
$i \in I(u,\nu)(\tau)$ and is implied by $k \in I(u,\nu)(\tau)$ for all $k$ with $d_k^\intercal u > d_j^\intercal u$.
We thus find the inclusion
\begin{align}
\bigcup_{l \in I(u,\nu)(\tau)} M_l(x, \eta) \big\backslash
\bigcup_{l \in I(u,\nu)(\tau)} M_l(x+u, \eta+\nu)
\supset M_i(x,\eta) \big\backslash \bigcup_{\{ k : d_k^\intercal u > d_j^\intercal u \}} M_k(x+u, \eta+\nu) \label{chain2}
\end{align}
where the set on the right-hand side does not depend on $\tau$ anymore. We want to estimate the $\mu$-measure of 
this set:\\ 

Remember that $K_i$ is a pointed cone and therefore has finitely many facets 
$\{F^i_j \}_{j=1,\ldots,f_i}$ and finitely many extreme-rays $\{ r t^{i,j}_k \ | \ k =1,\ldots,g_{i,j} \}$
adjacent to facet $F^i_j$.  For notational convenience set 
\begin{align*}
F^+_{i,j}(x,\eta) = (x+F^i_j) \cap \{ d_i^\intercal (z-x) \geq \eta \}.
\end{align*}
Intersecting $F^i_j$ with a hyperplane $\{ d_i^\intercal z = \eta^- \}$ - this really is a hyperplane since
$i \in I^+$ -  yields points $\{ y^j_1(\eta^-) , \ldots, y^j_{r_j}(\eta^-) \}$ where the hyperplane meets the extreme rays 
of $K_i$ adjacent to $F^i_j$:
\begin{align*}
\{ d_i^\intercal z = \eta^- \} \cap K_i = \{ y^j_1(\eta^-) \ldots, y^j_{r_j}(\eta^-) \}.
\end{align*}
Among all $y^j_k(\eta^-)$ pick some $\hat{y}^j(\eta^-)$ with
\begin{align}\label{y_hat_j}
\|\hat{y}^j(\eta^-) \| = \min\{ \| y^j_k(\eta^-) \|  \ | \ k =1,\ldots,r_j \} .
\end{align}

Choose $\eta^- > 0$ such that 
\begin{align}\label{eta_minus}
\max_{j = 1,\ldots, f_i}\| \hat{y}^j(\eta^-) \| < \rho,
\end{align}
which is possible since $\| \hat{y}^j(\eta^-) \| \rightarrow 0$ as $\eta^- \rightarrow 0+$,
and let $\tilde{\rho} = \rho - \max_{j = 1,\ldots, f_i}\| \hat{y}^j(\eta^-) \|$.

%Then we have
%\begin{align}\label{cyl_base}
%\beta = \lambda_{s-1}\big( F^+_{i,j}(0,\eta^-) \cap B_{\tilde{\rho}}( \hat{y}^j(\eta^-) ) \big) > 0.
%\end{align}
%Note that the function
%\begin{align*}
%\eta^- \mapsto 
% \lambda_{s-1}\big(F^+_{i,j}(0,\eta^-) \cap B_{\tilde{\rho}}( \hat{y}^j(\eta^-) ) \big) 
%\end{align*}
%is continuous, monotonically decreasing and tends to $\lambda_{s-1} \big(  F^i_j  \cap B_\rho(0) \big)$ as 
%$\eta \rightarrow 0+$.
For all $-\infty < \eta \leq \eta^-$ we can now show the inclusions
\begin{align}\label{incl1}
\bigcup_{0 \leq \lambda < 1} \Big( [ F^+_{i,j}(x,\eta^-) \cap B_{\tilde{\rho}}(x+\hat{y}^j(\eta^-)) ]+ \lambda u \big)
&\subset V + B_\rho(0), \\
\label{incl1_1}
\bigcup_{0 \leq \lambda < 1} \Big( [ F^+_{i,j}(x,\eta^-) \cap B_{\tilde{\rho}}(x+\hat{y}^j(\eta^-)) ]+ \lambda u \big)
&\subset \Big( M_i(x,\eta) \big\backslash \bigcup_{k:d_k^\intercal u > d_j^\intercal u} 
M_i(x+u, \eta+\nu) \Big).
\end{align}
To this end let $x + z + \lambda u$ with $z \in F^+_{i,j}(0,\eta^-) \cap B_{\tilde{\rho}}(\hat{y}^j(\eta^-) )$ and 
$0 \leq \lambda < 1$ be given. It holds $x + \lambda u \in V$ due to the convexity of $V$. Furthermore
\begin{align*}
\| z \| &\leq \| z - \hat{y}^j \| + \| \hat{y}^j(\eta^-) \| \\
         &\leq \tilde{\rho} + \| \hat{y}^j(\eta^-) \| \\
         &= \rho - \max_{k = 1,\ldots, f_i}\| \hat{y}^j_k(\eta^-) \| + \| \hat{y}^j(\eta^-) \| \leq \rho.
\end{align*}
It remains to be shown
\begin{align*}
x + z + \lambda u \in M_i(x,\eta) \big\backslash \bigcup_{\{ k : d_k^\intercal u > d_j^\intercal u \}} M_i(x+u, \eta+\nu).
\end{align*}
Due to $x+z \in F^+_{i_j}(x,\eta^-) \subset x + K_i$ and $u \in K_i$ we have $x+z+\lambda u \in x+K_i$, 
since $K_i$ is a convex cone. We also have 
\begin{align*}
d_i^\intercal( x + z + \lambda u - x) = d_i^\intercal z + \lambda d_i^\intercal u \geq \eta^- + 0 \geq \eta,
\end{align*}
which establishes $x+z+\lambda u \in M_i(x,\eta)$.
Let $k$ be any index in $I$ such that $d_k^\intercal u > d_j^\intercal u$.
We will show that $x+z+\lambda u$ does not even lie in $x+u+K_k$:
\begin{align*}
d_k^\intercal (x+z+\lambda u - x - u) &= d_k^\intercal z + (\lambda-1)d_k^\intercal u \\
                                      &< d_j^\intercal z + (\lambda-1)d_j^\intercal u = d_j^\intercal(x+z+\lambda u - x - u)
\end{align*}
where we used the fact that $z \in F^i_j$, implying $d_k^\intercal z \leq d_i^\intercal z = d_j^\intercal z$,
 and $\lambda < 1$.\\
 
As the last prerequisite step we want to show that there exists some constant $\beta > 0$
such that
\begin{align} \label{cyl_measure}
\lambda \Big(\bigcup_{0 \leq \lambda < 1} [ F^+_{i,j}(x,\eta^-) \cap B_{\tilde{\rho}}(x+\hat{y}^j(\eta^-)) ]+ 
\lambda u \Big) \geq \alpha \, \beta \, \|u\|.
\end{align}
First note that we may as well set $x=0$ due to the translation invariance of the Lebesgue measure. 
The set is then the Lebesgue measure of a cylindrical set with bases 
$F^+_{i,j}(0,\eta^-) \cap B_{\tilde{\rho}}( \hat{y}^j(\eta^-) )$
and $F^+_{i,j}(0,\eta^-) \cap B_{\tilde{\rho}}( \hat{y}^j(\eta^-) ) + u$.

Let
\begin{align}\label{cyl_base}
\beta = \lambda_{s-1}\big( F^+_{i,j}(0,\eta^-) \cap B_{\tilde{\rho}}( \hat{y}^j(\eta^-) ) \big) > 0.
\end{align}
This constant still depends on the index $j$ which in turn depends on 
the direction $u \in K_i$. Since there are only finitely many possible
choices of $j$ we can robustify by picking the minimal $\beta$.
 
The estimate \eqref{cyl_measure} then follows from \eqref{est1} and Cavalieri's principle.
As a side-remark - which comes into play when trying to maximize the constant of partial strong convexity - 
we note that the function
\begin{align*}
\eta^- \mapsto 
 \lambda_{s-1}\big(F^+_{i,j}(0,\eta^-) \cap B_{\tilde{\rho}}( \hat{y}^j(\eta^-) ) \big) 
\end{align*}
is continuous, monotonically decreasing and tends to $\lambda_{s-1} \big(  F^i_j  \cap B_\rho(0) \big)$ as 
$\eta \rightarrow 0+$.
 
We have gathered all necessary information to continue in \eqref{gradient_formula} as
\begin{align*}
&\int \mu\big( \bigcup_{l \in I(u,\nu)(\tau)} M_l(x, \eta) \big\backslash
\bigcup_{l \in I(u,\nu)(\tau)} M_l(x+u, \eta+\nu) \big) \, \text{d}\tau \\
&\overset{\eqref{chain2}}{\geq} \int_{-d_i^\intercal u - \nu}^{-d_j^\intercal u - \nu}
\mu\big( M_i(x,\eta) \big\backslash \bigcup_{\{ k : d_k^\intercal u > d_j^\intercal u \}} 
M_k(x+u, \eta+\nu) \big) \, \text{d}\tau \\
&\overset{\eqref{est1}}{\geq}\alpha \|u\| \: \mu\big( M_i(x,\eta) \big\backslash \bigcup_{\{ k : d_k^\intercal u > d_j^\intercal u \}} 
M_k(x+u, \eta+\nu) \big) \\
&\overset{\eqref{incl1_1}}{\geq} \alpha \|u\| \: \int_{\bigcup_{0 \leq \lambda < 1} 
 [ F^+_{i,j}(x,\eta^-) \cap B_{\tilde{\rho}}(x+\hat{y}^j(\eta^-)) ]+ \lambda u}
 \theta(z) \, \text{d}z \\
&\overset{\eqref{incl1}}{\geq} \alpha \, r \: \|u\| \: \lambda\big( \bigcup_{0 \leq \lambda < 1} 
\big( [ F^+_{i,j}(x,\eta^-) \cap B_{\tilde{\rho}}(x+\hat{y}^j(\eta^-)) ]+ \lambda u \big) \\
&= \alpha \, r \: \|u\| \: \lambda\big( \bigcup_{0 \leq \lambda < 1} 
\big( [ F^+_{i,j}(0,\eta^-) \cap B_{\tilde{\rho}}(\hat{y}^j(\eta^-)) ]+ \lambda u \big) \\
&\overset{\eqref{cyl_measure}}{\geq} \alpha^2 \, \beta \, r \, \|u\|^2.
\end{align*}

In the first inequality the nonnegativity of the integrand and in the only equality the translation invariance
of the Lebesgue measure was used.\\ \\

Choose now some $0 < \eta^- \leq \eta^+$ with 
\begin{align}\label{eta_plus}
\| \hat{y}^j(\eta^+) \| = \min_{k=1,\ldots,f_i} \| \hat{y}^j_k(\eta^+) \| < \rho
\end{align}
and set
\begin{align*}
\tilde{\rho} = \rho - \| \hat{y}^j(\eta^+) \|.
\end{align*}
Now consider the case $\eta^- < \eta \leq \eta^+$:\\
We will choose $\left(-d_i^\intercal u - \nu, 0 \right)$ as the area of integration in \eqref{gradient_formula}.
As the integration variable $\tau$ satisfies  $-d_i^\intercal u - \nu < \tau < 0$ we can find a subset of the one in
\eqref{gradient_formula} as
\begin{align}\label{another_inclusion_1}
\bigcup_{k \in I(u,\nu)(\tau)} M_k(x, \eta) \big\backslash
\bigcup_{k \in I(u,\nu)(\tau)} M_k(x+u, \eta+\nu)
\supset M_i(x,\eta) \big\backslash \bigcup_{k: -d_k^\intercal u - \nu < 0} M_k(x+u,\eta+\nu).
\end{align}

In analogy with the notation used in the previous case set 
\begin{align*}
F_{i,0}(x,\eta) = (x+K_i) \cap \{ d_i^\intercal (z-x) = \eta \}.
\end{align*}

It holds
\begin{align}\label{another_inclusion_2}
\bigcup_{0 \leq \lambda < 1}
\big( [ F_{i,0}(x,\eta) \cap B_{\tilde{\rho}}(x+\hat{y}^i(\eta))  ] + \lambda u \big)
&\subset V+B_\rho(0), \\
\label{another_inclusion_3}
\bigcup_{0 \leq \lambda < 1}
\big( [ F_{i,0}(x,\eta) \cap B_{\tilde{\rho}}(x+\hat{y}^i(\eta))  ] + \lambda u \big)
&\subset M_i(x,\eta) \big\backslash \bigcup_{k:-d_k^\intercal u - \nu < 0}M_k(x+u,\eta+\nu).
\end{align}
Inclusion \eqref{another_inclusion_2} follows the same way as \eqref{incl1}. The proof for \eqref{another_inclusion_3} 
works similar to the one for \eqref{incl1_1} with minor modifications:\\
Let $x+z+\lambda u$ with $z \in K_i \cap \{ d_i^\intercal z' = \eta \}$ and $0 \leq \lambda < 1$.
Again $x+z+\lambda u \in x+K_i$ due to $K_i$ being a convex cone.  Furthermore
\begin{align*}
d_i^\intercal (x+z+\lambda u - x) = d_i^\intercal z + \lambda d_i^\intercal u \geq \eta
\end{align*}
because $d_i^\intercal z = \eta$ and $d_i^\intercal u \geq 0$. Thus $x+z+\lambda u \in M_i(x,\eta)$ has been established.
Pick any index $k \in I$ with $-d_k^\intercal u - \nu < 0$. It holds
\begin{align*}
d_k^\intercal (x+z+\lambda u - x - u) = d_k^\intercal z + (\lambda - 1) d_k^\intercal u
< \eta +  (\lambda - 1) d_k^\intercal u < \eta + (1-\lambda) \nu \leq \eta +\nu,
\end{align*}
so we have $x+z+\lambda u \notin M_k(x+u,\eta+\nu)$. In the last inequality we have used $\nu \geq 0$. \\

As the last ingredient we will show that there exist constants $\alpha', \beta' > 0$
\begin{align}\label{cyl_measure_2}
\lambda\big( \bigcup_{0 \leq \lambda < 1}
\big( [ F_{i,0}(0,\eta^-) \cap B_{\tilde{\rho}}(\hat{y}^i(\eta^-))  ] + \lambda u \big) \big)  
\geq \alpha' \, \beta' \, \|u\|
\end{align}
To this end set 
\begin{align}\label{cyl_base_2}
\beta' = \lambda_{s-1}\big( F_{i,0}(0, \eta^-) \cap B_{\tilde{\rho}}(\hat{y}^i(\eta^-) )  \big) > 0
\end{align}
which is positive by construction. Applying Lemma \ref{lemma1} (i) and Cavalier's principle yields \eqref{cyl_measure_2}. 
Having a lower bound on the measure of the set $ F_{i,0}(0, \eta^-) \cap B_{\tilde{\rho}}(\hat{y}^i(\eta^-) )$ 
is the reason why we need to have to distinct the cases $-\infty < \eta < \eta^-$ and $ \eta^- \leq \eta \leq  \eta^+$ 
in the first place. 

With this we can continue \eqref{gradient_formula} as
\begin{align*}
\eqref{gradient_formula} &\overset{\eqref{another_inclusion_1}}{\geq} \int_{-d_i^\intercal - \nu}^{0} 
\mu\big( M_i(x,\eta) \big\backslash
\bigcup_{k: -d_k^\intercal u - \nu < 0}M_k(x+u,\eta+\nu)  \big) \, \text{d}\tau  \\
                      &\overset{\eqref{est2}}{\geq} \, \|u\| \;  \mu\big( M_i(x,\eta) \big\backslash
\bigcup_{k: -d_k^\intercal u - \nu < 0}M_k(x+u,\eta+\nu)  \big) \\
                      &\overset{\eqref{another_inclusion_2}}{\geq} \alpha' \, \|u\| \; \mu\big( \bigcup_{0 \leq \lambda < 1}
\big( [ F_{i,0}(x,\eta) \cap B_{\tilde{\rho}}(x+\hat{y}^j(\eta))  ] + \lambda u \big) \big) \\
                      &= \alpha' \, \|u\| \; \mu\big( \bigcup_{0 \leq \lambda < 1}
\big( [ F_{i,0}(0,\eta^-) \cap B_{\tilde{\rho}}(\hat{y}^j(\eta^-))  ] + \lambda u \big) \big) \\
                     &\overset{\eqref{another_inclusion_3}}{\geq} \alpha' \, r \, \|u\| \; \lambda\big( \bigcup_{0 \leq \lambda < 1} 
\big( [ F_{i,0}(0,\eta^-) \cap B_{\tilde{\rho}}(\hat{y}^j(\eta^-))  ] + \lambda u \big) \big) \\
                    &\overset{\eqref{cyl_measure_2}}{\geq} \alpha'^2 \, r \, \beta' \, \|u\|^2.
\end{align*}
As before, nonnegativity of the integrand is used in the first step. Translation invariance of the Lebesgue measure is exploited in the only equation.\\

All constants $\alpha,\alpha',\beta$ and $\beta'$ computed until now implicitly depend on the index $i$ for which we have $u \in K_i$. But since
the index set $I$ is finite so is the number of such constants. Choosing minimal constants thus concludes the proof of the first 
part.
\qed
\end{proof}

That fact that the sets  $M_l(x,\eta)$ have a simple geometry, i.e. they are pointed cones or 
truncated cones, is one of the key arguments in the proof (cf. Fig. 2 below).
This geometry allowed us to explicitly construct cylindrical sets (\eqref{incl1_1} resp. \eqref{another_inclusion_2}) for the estimates needed after \eqref{gradient_formula}.
The size of these cylindrical sets is dictated by the model data, e.g. $\rho$, and a certain degree of freedom when choosing
$\eta^-$ and $\eta^+$. With a little more effort the constants $\beta$ and $\beta'$ which depend on the choice of
$\eta^-$ and $\rho$ explicitly (and on $d_1,\ldots,d_N$ implicitly) can be maximized to yield a partial strong 
convexity constant as large as possible. In principle the modulus of partial strong convexity resp. of restricted partial
strong convexity (in the next theorem) can thus be computed directly in terms of model data. 
The remarks after Theorem \ref{theorem0}. also apply in the setting of Theorem \ref{theorem2}.\\
As a last remark to the preceding theorem we point out that once a $\eta^- > 0$ has been fixed, the arguments
in the proof can be modified to show that A6 implies strong convexity of $Q_{EE}$ in both arguments with strong convexity constant depending on $\eta^-$.

\begin{figure}
  \caption{Truncated cones for varying $\eta$}
  \centering
    \includegraphics[width=1.0\textwidth, trim={0 6cm 0 6cm},clip]{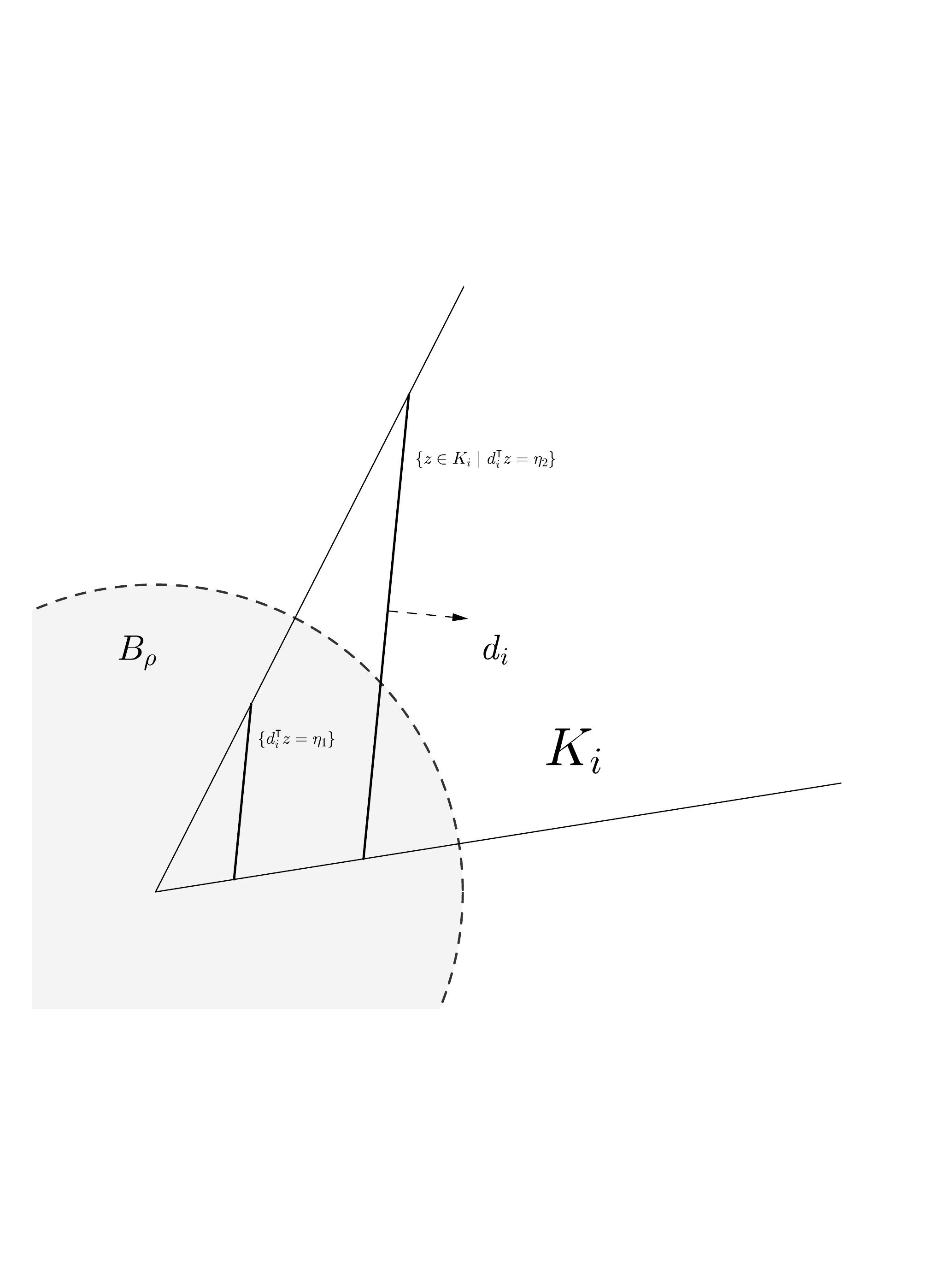}
\end{figure}

Next we will consider the more general case when A6 fails to hold and prove restricted partial strong convexity of 
$Q_{EE}$ wrt. the first argument in the sense of Definition \ref{def2}. In the last theorem $u$ and $\nu$ varied independently
of each other, but as example 3 shows, we need to make some new assumptions which tie together $u$ and $\nu$ 
in order for restricted partial strong convexity to hold. This necessitates more case distinctions and technicalities, mainly due
to the following two facts:
 \begin{itemize}
\item[1.] The interplay between choosing the area of integration in \eqref{gradient_formula} and 
constructing suitable subsets of the set in the integrand become more subtle.
\item[2.] The geometry of the sets $M_i^+$ can be slightly more complicated. 
\end{itemize}

To avoid being overly repetitive in the proof of the next theorem, we will borrow notation from the preceeding one and 
mostly point out where changes need to be made to the last proof to accomodate the new situation,
i.e. where two main consequences of A6 - being able to reflect to $\nu \geq 0$ when necessary and the lower 
bound on $d_i^\intercal u$ from \eqref{est2} - were used.

We feel that it is still convenient to separate the two theorems: Firstly, to not obscure
the general structure of the proof by even more case distinctions. Secondly, because most of the discussions in section 3
only make use of Theorem \ref{theorem2} anyway. \\

We start with some simple geometric observations which apply when A6 is discarded:\\
Before Theorem \ref{theorem0}. we introduced the notation $K_i^+$ and $I^+$. For $\eta > 0$ we
had observed that for $i \in I^+$ the hyperplane $\{ d_i^\intercal z = \eta \}$ intersects $K_i$ (and also $K_i^+$) 
in some of its extreme rays rays. We denote a point of intersection (say with $\{ r \, t^i \ | \ r\geq 0 \}$) having minimum
norm among all such points as $\hat{y}^i$. Assume that $t^i$ has norm $1$ and set $\gamma = d_i^\intercal t^i$. 
We then have the estimate 
\begin{align}\label{slice}
d_i^\intercal u \leq \gamma \|u\| \text{ for all $u \in K_i^+$.}
\end{align}

The hyperplane $\{ d_i^\intercal z = \eta \}$ also slices $K_i$ into two polyhedra 
$M_i^+(0,\eta) =K_i \cap \{ d_i^\intercal z \geq \eta \}$
and $M_i^-(0,\eta) = K_i \cap \{ d_i^\intercal z \leq \eta \}$. Denote with $I^\pm$ indices $i \in I^+$ such that 
both polyhedra are unbounded - it is $I^\pm \neq \emptyset$ if and only if A6 fails to hold - 
and with $I^{++}$ indices in $I^+$ such that only $M_i^+(0,\eta)$ is unbounded  
- which holds iff there is some $\alpha'>0$ with 
\begin{align}\label{slice3}
d_i^\intercal u \geq \alpha' \|u\| \text{ for all $u \in K_i$. }
\end{align}
Obviously $I^{++} \cap I^\pm = \emptyset$ and $I^{++} \cup I^\pm = I^+$.\\
For $i \in I^\pm$ we note that by inequality \eqref{slice} we can write
\begin{align}\label{slice2}
K_i^+ = K_i^{++} \cup K_i^\pm
\end{align}
with two full-dimensional polyhedral cones $K_i^{++}  = K_i^+ \cap \{ u \ | \ d_i^\intercal u \geq \gamma' \|u\| \}$
and $K_i^\pm  = K_i^+ \cap \{ u \ | \ d_i^\intercal u \leq \gamma' \|u\| \}$ 
(choose for example $\gamma' = \frac{\gamma}{2}$).

\begin{Theorem}[Restricted partial strong convexity of $Q_{EE}$]\label{theorem3}
Let A1-A5 hold.\\Then $Q_{EE}(x,\eta)$ is restricted partially strongly convex on the set 
\begin{align*}
\Omega = \{ (x,\eta) \in V_{\eta_0} \ | \ \eta' - \eta \geq -\frac{\delta}{3} \| y-x \| \ \forall (y,\eta') \in V_{\eta_0} \ : 
\ y-x \in K_i, \ i \in I^+\},
\end{align*}
where $\delta = \max\{ \alpha'_i, \gamma_j' \ | \ i \in I^{++}, j \in I^\pm \}$ with $\alpha'_i$ and $\gamma_j'$ from 
\eqref{slice3} and \eqref{slice2}. For the definition of $V_{\eta_0}$ cf. Theorem \ref{theorem2} above.
\end{Theorem}

\begin{proof}
Let $(x,\eta),(x+u, \eta +\nu) \in \Omega$ so that we have 
\begin{align}\label{nu_condition}
\nu \geq -\frac{\delta}{3} \|u\|
\end{align}
as required by the definition of $\Omega$.
In \eqref{SCformula} we may, after a suitable change of variables, 
assume that $u \in K_i^+$ with $i \in I^+ = I^{++} \cup I^\pm$.
We can however not take $\nu \geq 0$ as granted.
Since the case $i \in I^{++}$ is structurally more similar to the ones already treated, we shall start
with this one.  Drawing on condition A5 choose again 
$0 < \eta^- \leq \eta^+$ so that conditions \eqref{eta_minus} and \eqref{eta_plus} are fulfilled.
Let us first consider the case  $\eta^- \leq \eta \leq \eta^+$:\\
We need to choose the area of integration in \eqref{gradient_formula} differently as before:
This time it shall be
\begin{align*}
-d_i^\intercal u - \nu \leq \tau < -d_i^\intercal u - \nu + \frac{\alpha}{3}\|u\|.
\end{align*}
Consequently we need to replace \eqref{another_inclusion_1} by 
\begin{align*}
\bigcup_{k \in I(u,\nu)(\tau)} M_k(x, \eta) \big\backslash
\bigcup_{k \in I(u,\nu)(\tau)} M_k(x+u, \eta+\nu) 
\supset M_i(x,\eta) \big\backslash \bigcup_{k: -d_k^\intercal u < \frac{\alpha}{3}\|u\| - d_i^\intercal u} M_k(x+u,\eta+\nu).
\end{align*}
Inclusions \eqref{another_inclusion_2} and \eqref{another_inclusion_3} must be replaced by
\begin{align}
\bigcup_{0 \leq \lambda < \frac{1}{2}}
\big( [ F_{i,0}(x,\eta) \cap B_{\tilde{\rho}}(x+\hat{y}^i(\eta))  ] + \lambda u \big)
&\subset V+B_\rho(0) \label{another_another_inclusion_2}\\
\bigcup_{0 \leq \lambda < \frac{1}{2}}
\big( [ F_{i,0}(x,\eta) \cap B_{\tilde{\rho}}(x+\hat{y}^i(\eta))  ] + \lambda u \big)
&\subset M_i(x,\eta) \big\backslash \bigcup_{k: -d_k^\intercal u < \frac{\alpha}{3}\|u\| - d_i^\intercal u} M_k(x+u,\eta+\nu)
\label{another_another_inclusion_3}
\end{align}
where only
\begin{align*}
\bigcup_{0 \leq \lambda < \frac{1}{2}}
\big( [ F_{i,0}(x,\eta) \cap B_{\tilde{\rho}}(x+\hat{y}^i(\eta))  ] + \lambda u \big) \bigcap
\bigcup_{k: -d_k^\intercal u < \frac{\alpha}{3}\|u\| - d_i^\intercal u} M_k(x+u,\eta+\nu) = \emptyset
\end{align*}
needs justification:\\ For any $0 \leq \lambda < \frac{1}{2}$, $z \in F_{i,0}(x,\eta)$ and index $k \in I$ with
$-d_k^\intercal u < \frac{\alpha}{3}\|u\| - d_i^\intercal u$ it holds
\begin{align*}
d_k^\intercal (z + \lambda u - x - u) &= d_k^\intercal (z-x) + (1-\lambda)(-d_k^\intercal u) 
\leq d_i^\intercal (z-x) + (1-\lambda)(-d_k^\intercal u) \\ 
&= \eta + (1-\lambda)(-d_k^\intercal u) <  \eta + (1-\lambda) ( \frac{\alpha}{3}\|u\| - d_i^\intercal u) \\
& <  \eta + (1-\lambda) ( \frac{\alpha}{3}\|u\| - \alpha \|u\|) = \eta + (1-\lambda) ( -\frac{2 \alpha}{3}\|u\|) \\
& \leq \eta - \frac{1}{3} \alpha \|u\| \leq \eta + \nu.
\end{align*}
With \eqref{another_another_inclusion_2} and \eqref{another_another_inclusion_3} at hand the remaining estimates
are analogous to the ones made before. The case $-\infty < \eta \leq \eta^-$ is identical to the one in (i).  \\

For $i \in I^\pm$ we shall construct $\eta^-$ a little different than before.\\
Let us also assume that the set $K_i^- = K_i \cap \{ z \ | \ d_i^\intercal z \leq 0 \}$ has nonempty interior. 
The other case can be handled in a similar way. \\
We then find that for $\eta < 0$ the hyperplane $\{d_i^\intercal z = \eta \}$ intersects the extreme rays of
$K^-$ in singletons, let $\hat{y}^i = \hat{y}^i(\eta)$ denote one point of intersection with minimum norm. 
Choose $\eta^-$ < 0 so that $\| \hat{y}^i(\eta^-) \| < \rho$ and set $\tilde{\rho} = \rho - \| \hat{y}^i(\eta^-) \|$.\\
For arbitrary $-\infty < \eta \leq \eta^-$ we see that for all $j$ with $d_j$ adjacent to $d_i$ it holds
\begin{align*}
(x+F^i_j) \cap B_{\| \hat{y}^i(\eta^-) \|}(x) \subset F^+_{i,j}(x,\eta) 
\end{align*}
and 
\begin{align*}
\lambda\big( (x+F_{i,j}) \cap B_{\| \hat{y}^i \|}(x) \big) > 0.
\end{align*}
there first inclusion holding true since $\| \hat{y}^i \|(\eta)$ is monotonically increasing in $\eta$, the second
one bcause $\eta^- < 0$. \\

With this and resorting to \eqref{est1} we can be continue \eqref{gradient_formula} using as area of integration 
$-d_i^\intercal u -\nu < \tau < -d_j^\intercal u -\nu$. In \eqref{incl1} and \eqref{incl1_1}
the left hand side needs to be replaced by
\begin{align*}
\bigcup_{0 \leq \lambda < 1} \big( (x+F^i_j) \cap B_{\| \hat{y}^i(\eta^-) \|}(x)+ \lambda u \big),
\end{align*}
everything else is straightforward from thereon.\\

Now consider $\eta^- \leq \eta \leq \eta^+$ and - employing \eqref{slice2} - look at the cases\\
$u \in K_i^+ \cap \{ d_i^\intercal u' \leq \gamma' \|u'\| \}$
and $u \in K_i^+ \cap \{ d_i^\intercal u' \geq \gamma' \|u'\| \}$ separately.\\ 
For each of the two cases the area of integration in \eqref{gradient_formula} and estimates for the integrands
(as seen in \eqref{chain2}, \eqref{incl1} and \eqref{incl1_1}) must be done appropriately as demonstrated before.
\qed
\end{proof}

\section{CVaR based models}
We shall now discuss implications of the preceding results for models extending \eqref{OptProblem1}
by replacing the expectation-functional with the the conditional value at risk $\alpha CVaR$:
\begin{align}\label{OptProblem3}
\min\{ \ \alpha CVaR[ \, h(\xi) +  \varphi(z(\omega) - T\xi) \, ] \ | \ \xi \in X \}.
\end{align}
By translation equivariance of $\alpha CVaR$ and the same arguments as above we can rewrite
this problem as
\begin{align}\label{OptProblem4}
\min \{\hat{h}(x) + Q_{\alpha CVaR}(x) \ | \ x \in T(X) \}
\end{align}
with
\begin{align}\label{QCVaR}
Q_{\alpha CVaR}(x) = \alpha CVaR[ \, \varphi(z-x) \, ].
\end{align}
Theorem \ref{theorem2} and the discussion after Lemma \ref{MonMap} yields 
\begin{theorem}[Strong convexity of $Q_{\alpha\,CVaR}$] \label{ThCVaR}
Assume A1-A6 (in particular, there is some $\eta_0 > 0$ satisfying A5) 
and the following condition
\begin{equation}
\label{CVaRCond}
\sup_{x \in V} Q_{\alpha VaR}(x) \leq \eta_0.
\end{equation}
Then $Q_{\alpha \, CVaR}$ is $\frac{\kappa}{\alpha}$-strongly convex on $V$
with $\kappa$ being the modulus of partial strong convexity for $Q_{EE}$ for $\eta \leq \eta_0$.
\qed
\end{theorem}

Let us make some remarks on this theorem:\\
Since $\alpha \mapsto Q_{\alpha VaR}(Y)$ is nonincreasing for fixed $Y$, condition \eqref{CVaRCond} will hold \\ 
for all $\alpha \leq \alpha' \leq 1$ if it holds for $\alpha$. It follows that $Q_{\alpha' CVaR}$ is strongly convex
for all such $\alpha'$.\\

There is some heuristic on when one can
hope for $Q_{\alpha \, CVaR}$ to be strongly convex:
We have that $Q_{\alpha CVaR} \equiv Q_\mathbb{E}$ for $\alpha = 1$ which
is strongly convex given the usual assumptions made above.
If $1 \geq \alpha \geq \alpha_0$ for some $\alpha_0$ which is not too 
close to $0$ condition \eqref{CVaRCond} might still hold.
When $\alpha \rightarrow 0+$ the quantity $Q_{\alpha VaR}(x)$ will
increase and condition \eqref{CVaRCond} might be violated.\\ \\
If not on $V$ it might still be possible to show strong convexity on some subset $U$ of $V$ for two reasons:
Firstly, on $U$ conditions A5 is weaker since a larger $\rho$ can be chosen so that condition A5 holds.
Secondly, in \eqref{CVaRCond} we have the obvious estimate  
$\sup_{x \in V} Q_{\alpha VaR}(x) \geq \sup_{x \in U} Q_{\alpha VaR}(x)$.
We give an academic example to illustrate these points in the appendix, cf. example 5 there.\\

If $0 \in V$, the following (very rough) upper bound for the value-at-risk might also be helpful: Set $d = \max_{i \in I} \|d_i\|$, then for any $x \in U \subseteq V$ and $z \in \mathbb{R}^s$ we have
$$
\varphi(z-x) = \max_{i=1,\ldots, N} d_i^\intercal(z-x) \leq d\|z\| + d \max_{x \in U} \|x\|.
$$
Thus,
\begin{align*}
Q_{\alpha VaR}(x) \; &\leq \;   \inf \Big\{ t \in \mathbb{R} \; | \; \mu(\lbrace z \; | \; \|z\| \leq \frac{t}{d} - \max_{x \in U} \|x\| \rbrace) \geq \alpha \Big\} \\
&= \; d \inf \Big\{ t \; | \; \mu(B_t(0)) \geq \alpha \Big\} + d \max_{x \in U} \|x\|.
\end{align*}
The above quantity is finite by the tightness of the probability measure $\mu$ and the boundedness of $U$. Set $\bar{\eta} := \inf \left\{ t \; | \; \mu(B_t(0)) \geq \alpha \right\}$. As direct consequence of the above considerations, we obtain the following:

\begin{proposition}
If $Q_{EE}(\cdot;d \bar{\eta} + \epsilon)$ is $\kappa$-strongly convex on $V$ for some $\epsilon > 0$, then $Q_{\alpha CVaR}$ is strongly convex with modulus $\frac{\kappa}{\alpha}$ on any nonempty, open, convex $U \subseteq V$ satisfying
$$
\max_{x \in U} \|x\| \leq \frac{\epsilon}{d}.
$$
\end{proposition}

As a consequence of theorem \ref{theorem3} we get 
\begin{proposition}
Assume A1-A5, \eqref{CVaRCond} and the condition
\begin{align*}
\label{LipschitzCond}
| Q_{\alpha VaR}(y) - Q_{\alpha VaR}(x) | \leq \frac{\delta}{3} \| y-x \|
\end{align*}
for all $x,y \in V$ with $\delta$ as defined in theorem \ref{theorem3}. 
Then $Q_{\alpha \, CVaR}$ is $\frac{\kappa}{\alpha}$-strongly convex on $V$
with $\kappa$ being the modulus of partial strong convexity for $Q_{EE}$ for $\eta \leq \eta_0$.
  
\end{proposition}

\subsection{Coherent risk measures and spectral risk measures}
With verifiable conditions for strong convexity of CVaR-based models at hand,
we shall now consider risk measures that can be represented as mixtures
of CVaRs, so called coherent risk measures. For a general discussion of such 
functionals we refer to \cite{ADEH99} and \cite{FS11}.
\begin{definition}
Let $\mathcal{Z} = L_1(\Omega, \mathcal{F}, \mathbb{P})$. A proper function
$\rho: \mathcal{Z} \rightarrow \overline{\mathbb{R}}$ is called a coherent risk 
measure if it satisfies the following four properties:
\begin{itemize}
\item[(1)] $\rho(tZ + (1-t)Z) \leq t\rho(Z) + (1-t)\rho(Z)$ for all 
$Z \in \mathcal{Z}, \, 0 \leq t \leq 1$. (Convexity)
\item[(2)] $\rho(Z) \leq \rho(Z')$ for $Z, Z' \in \mathcal{Z}$ such that 
$Z \leq Z'$ holds $\mathbb{P}$-almost surely. (Monotonicity)
\item[(3)] $\rho(a+Z) = a + \rho(Z)$ for all $Z \in \mathcal{Z}, \, a \in \mathbb{R}$.
(Translation equivariance)
\item[(4)] $\rho(t\,Z) = t \rho(Z)$ for all $Z \in \mathcal{Z}, \, t>0$. (Positive homogeneity)
\end{itemize}
\end{definition}

\begin{theorem}[Kusuoka \cite{Kusuoka2001}]
Assume $(\Omega,\mathcal{F}, \mathbb{P})$ is nonatomic and 
$\rho$ is a law-invariant, coherent risk measure on $\mathcal{Z}$. 
Then we have for any $Z \in \mathcal{Z}$
\begin{equation}\label{Kusuoka}
\rho(Z) = \underset{\nu \in \mathfrak{M}}{\sup} \int CVaR_{\alpha}(Z) \, 
\nu(\text{d}\alpha)
\end{equation}
where $\mathfrak{M}$ is a set of probability measures on the interval
$[0,1)$.

\end{theorem}
As in the preceding paragraphs we consider as probability space
$(\Omega,\mathcal{F},\mathbb{P}) = (\mathbb{R}^s, \mathcal{B}(\mathbb{R}^s), \mu)$
which clearly is non-atomic due to $\mu$ having a Lebesgue-density.
Now consider the random variables $Z_x(z) = \varphi(z-x)$ and the induced functional
\begin{align*}
Q_\rho: \mathbb{R}^s \supset V \rightarrow \mathbb{R}, \ \ Q(x) = \rho \, (Z_x).
\end{align*}
This gives the induced Kusuoka representation
\begin{align}\label{Kusuoka2}
Q_\rho(x) = \underset{\nu \in \mathfrak{M}}{\sup} \int Q_{\alpha\,CVaR}(x) \, 
\nu(\text{d}\alpha).
\end{align}
Formula \eqref{Kusuoka2} makes theorem \ref{ThCVaR} applicable to derive sufficient conditions
 for strong convexity of $Q_\rho$ if information on $\mathfrak{M}$ is available.
In the scope of this paper we shall only investigate comonotonic risk measures appearing 
in the Kusuoka representation as
\begin{align*}
Q_{\nu}(x) = \int Q_{\alpha\,CVaR}(x) \, \nu(\text{d}\alpha)
\end{align*}
for some probability measure $\nu$ on $\left( 0,1 \right]$ (cf. \cite[Theorem 2]{Shapiro2013}). In particular, we shall consider measures $\nu_g$ induced by continuous, increasing, concave
distortion functions $g: \left[0,1\right] \rightarrow \left[0,1\right]$ with $g(0) = 0, g(1) = 1$. 
These are defined on half open intervals as
\begin{align}
\label{Distortion1}
\nu_g( \left( 0, t \right] ) = g(t) - t g'(t)
\end{align}
if $t \in \left( 0,1 \right)$ and
\begin{align}
\label{Distortion2}
\nu_g( \{ 1 \} ) = \lim_{t \rightarrow 1-}g'(t)
\end{align}
else (cf. \cite{BK17} and \cite{FS11}, section 4.6). For a comprehensive treatment of distortion functions and risk measures we refer to \cite{BGM09}, \cite{Pflug2006} and \cite{Tsukahara09}. Theorem \ref{ThCVaR} yields criteria for strong
convexity of $Q_{\nu}$ and $Q_{\nu_g}$:
\begin{corollary}[Strong convexity for comonotonic risk measures] \label{CorDistortion}
Assume that\\$Q_{\alpha_0 \, CVaR}$ is strongly convex on some nonempty, open convex set $V$ with modulus of strong convexity $\kappa$ for some $0 < \alpha_0 < 1$ and that the
following inequality is fulfilled:
\begin{equation}
\label{ComonotonicCond}
c = \nu((\alpha_0,1]) > 0.
\end{equation}
Then $Q_\nu$ is strongly convex on $V$ with modulus $\kappa \, c$. If $\nu = \nu_g$ is generated by a distortion function, condition \eqref{ComonotonicCond} is equivalent to
\begin{align}
\label{DistortionCond}
c = 1 - g(\alpha_0) + \alpha_0 g'(\alpha_0) > 0.
\end{align}
\end{corollary}
\begin{proof}
Let $x,y \in V$ and $0 < \lambda < 1$. By splitting the integral into two 
and using convexity resp. strong convexity of the integrands we get
\begin{align*}
&Q_{\nu}(\lambda x + (1-\lambda)y ) = \int_{\left(0,1\right]} Q_{\alpha CVaR}(\lambda x + (1-\lambda) y) \, 
\nu(\text{d}\alpha) \\
&= \int_{\left( 0,\alpha_0 \right]} Q_{\alpha CVaR}(\lambda x + (1-\lambda) y) \, 
\nu(\text{d}\alpha)
+ \int_{\left( \alpha_0,1 \right]} Q_{\alpha CVaR}(\lambda x + (1-\lambda) y) \, 
\nu(\text{d}\alpha) \\
&\leq \lambda \int_{\left( 0,\alpha_0 \right]} Q_{\alpha CVaR}(x) \, \nu(\text{d}\alpha) 
+ (1-\lambda) \int_{\left( 0,\alpha_0 \right]} Q_{\alpha CVaR}(y) \, \nu(\text{d}\alpha) \\
&+ \lambda \int_{\left( \alpha_0,1 \right]} Q_{\alpha CVaR}(x) \, \nu(\text{d}\alpha) 
+ (1-\lambda) \int_{\left( \alpha_0,1 \right]} Q_{\alpha CVaR}(y) \, \nu(\text{d}\alpha) \\ &- \frac{\kappa}{2} \lambda (1-\lambda) \|x-y\|^2 \nu( \left(\alpha_0,1 \right]) \\
&= \lambda \int_{\left( 0,1 \right]} Q_{\alpha CVaR}(x) \, \nu(\text{d}\alpha)
+ (1-\lambda) \int_{\left( 0,1 \right]} Q_{\alpha CVaR}(y) \, \nu(\text{d}\alpha) \\
&- \frac{\kappa}{2} \lambda (1-\lambda) \|x-y\|^2 \nu( \left(\alpha_0,1 \right]) \\
&\geq \lambda Q_{\nu}(x) + (1-\lambda) Q_{\nu}(y) - \frac{\kappa \, c}{2} \lambda (1-\lambda) \|x-y\|^2.
\end{align*}
For distortion risk measures we have
\begin{align*}
\nu_g( \left( \alpha_0,1 \right] ) &= \nu_g( \left( 0,1 \right] \backslash 
\left(0,\alpha_0 \right] ) = \nu_g( \left( 0,1 \right]) - \nu_g(\left(0,\alpha_0 \right]) \\
&= 1 - ( g(\alpha_0) - \alpha_0 \, g'(\alpha_0) )   
=  1 - g(\alpha_0) + \alpha_0 \, g'(\alpha_0) = c > 0, 
\end{align*}
which completes the proof.
\qed
\end{proof}

To illustrate Corollary \ref{CorDistortion}, we shall discuss condition \eqref{DistortionCond} for various distortion functions. \\

The expectation is generated by the distortion function $g_\mathbb{E}(t) := t$. By $1 - g(t) + tg'(t) = 1$ for all $t$, condition \eqref{DistortionCond} is fulfilled with $c=1$. However, the assumption of strong convexity of $Q_{\alpha_0 CVaR}$ for some $\alpha_0 \in (0,1)$ is generally more restrictive than the assumptions of Theorem \ref{TheoremSCExp}.\\

The distortion function associated with the conditional value at risk $Q_{\alpha \, CVaR}$ is defined by  $g_{\alpha \, CVaR}(t) := \min \lbrace \frac{t}{\alpha}, 1 \rbrace$. We have
$$
1 - g_{\alpha \, CVaR}(t) + t g_{\alpha \, CVaR}'(t) = \bigg\{ \begin{matrix}
1, \; \; \text{for} \; t \leq \alpha \\
0, \; \; \text{for} \; t > \alpha
\end{matrix}.
$$
Thus, \eqref{DistortionCond} does not constitute an additional assumption for the conditional value at risk. \\

The Wang Transform distortion is given by $g_{\beta \, WT}(t) := \Phi(\Phi^{-1}(x) - \beta)$, where $\beta > 0$ is a parameter and $\Phi$ denotes the cdf. of the standard normal distribution (cf. \cite{Wang2000}).
For arbitrary $t \in (0,1)$, we calculate
$$
1 - g_{\beta \, WT}(t) + t g_{\beta \, WT}'(t) = 1 - \underbrace{g_{\beta \, WT}(t)}_{\leq 1} + t \underbrace{\exp(- \beta t - \frac{1}{2} \beta^2)}_{> 0} > 0.
$$
Consequently, condition \eqref{DistortionCond} is always fulfilled for the Wang Transform. \\

The mappings $g_{\gamma \, PH}(t) := t^\gamma$ with $\gamma \in (0,1]$ form the parametrized familiy of Proportional Hazard distortion functions. For any feasible $\gamma$ and any $t \in (0,1)$, we have
$$
1 - g_{\gamma \, PH}(t) + t g_{\gamma \, PH}'(t) = 1 - (1 - \gamma)t^\gamma > 0,
$$
which means that condition \eqref{DistortionCond} holds for any Proportional Hazard distortion function. \\

The Lookback distortion is given by $g_{\gamma \, LB}(t) := t^\gamma (1 - \gamma \ln(t))$, where $\gamma \in (0,1]$ is a parameter. For any $t \in (0,1)$, we calculate
$$
1 - g_{\gamma \, LB}(t) + t g_{\gamma \, LB}'(t) = 1 - t^\gamma (1 + \gamma \ln(t) (1 - \gamma)) > 0.
$$
Thus, condition \eqref{DistortionCond} is fulfilled.

\section{Stability}

While we have only considered $Q_{\alpha \, CVaR}$ and $Q_{\nu_g}$ as functions of the first-stage decision variable $x$ so far, these quantities also depend on the underlying probability measure $\mu$. In stochastic programming, incomplete information about the true underlying distribution or the need for computational efficiency may lead to optimisation models that employ an approximation of $\mu$. Stability analysis deals with the behaviour of optimal values and optimal solution sets of the perturbed models in comparison to the original one.

\medskip   

First, we shall recall some relevant results concerning stability and strong convexity of abstract parametric programs of the form
\begin{equation}
\leqnomode
\label{Pt}
\tag*{$\textbf{P(t)}$}
\min_x \lbrace g(x,t) \; | \; x \in X(t) \rbrace
\end{equation}
where $g: \mathbb{R}^n \times T \to \mathbb{R}$ and $X: \mathcal{T} \rightrightarrows \mathbb{R}^n$ are functions and $t$ varies in some metric space $(\mathcal{T},\mathrm{d}_\mathcal{T})$. With inequality constraints and differentiable data, stability analysis for \eqref{Pt} goes back to Alt \cite{Alt1983}, while a more general setting is considered by Klatte in \cite{Klatte1987}. For constant feasible set $X(t) \equiv X \subseteq \mathbb{R}^n$ for all $t \in \mathcal{T}$, a proof of the following result is also given in \cite{RoemischSchultz1989}.

\begin{lemma} \label{StabilityKlatte}
Let $X$ be some nonempty, closed, convex subset of $\mathbb{R}^n$ and consider the mapping $\Psi: T \rightrightarrows \mathbb{R}^n$ given by
$$
\Psi(t) := \mathrm{Argmin}_x \lbrace g(x,t) \; | \; x \in X \rbrace.
$$
Suppose that $t_0 \in \mathcal{T}$ is such that the following conditions are satisfied:
\begin{enumerate}
\item[S1] $g(\cdot,t)$ is convex for all $t$ in a neighborhood of $t_0$.
\item[S2] There exists a bounded open set $Q \subset \mathbb{R}^n$ such that $\Phi(t_0) \subseteq Q$.
\item[S3] There exist $x_0 \in \Psi(t_0)$ and $\alpha: \mathbb{R}^n \to [0,\infty)$ such that $\alpha(0) = 0$ and
\begin{equation}
\label{LocalGrowth}
g(x,t_0) \geq g(x_0,t_0) + \alpha(x-x_0) \; \forall x \in X \cap \mathrm{cl} \; Q.
\end{equation}
\item[S4] There exists a constant $L > 0$ such that
\begin{equation}
\label{LipschitzCond}
|g(x,t) - g(x,t_0)| \leq L \mathrm{d}_T(t,t_0)
\end{equation}
holds for all $x \in \mathrm{cl} \; Q$ and all $t$ in a neighborhood of $t_0$.
\end{enumerate}
Then there exists a neighborhood of $t_0$ on which we have $\Psi(t) \neq \emptyset$ and
\begin{equation}
\label{KlatteEstimate}
\sup_{x \in \Psi(t)} \alpha(x-x_0) \leq 2 L \mathrm{d}_\mathcal{T}(t,t_0).
\end{equation}
\end{lemma}

In the presence of strong convexity, assumptions S2 - S4 above can be weakened.

\begin{lemma} \label{LemmaStronglyConvexStability}
Let $X \subseteq \mathbb{R}^n$ be nonempty, closed and convex and suppose that $t_0 \in \mathcal{T}$ is such that S1 and the following conditions are satisfied:
\begin{enumerate}
\item[C1] $g(\cdot,t_0)$ is $\kappa$-strongly convex on some open convex set $V$ with $\Psi(t_0) \cap V \neq \emptyset$.
\item[C2] There is a constant $L > 0$ such that \eqref{LipschitzCond} holds for all $t$ in a neighborhood of $t_0$ and all $x$ in a neighborhood of $\Psi(t_0)$.
\end{enumerate}
Then there exists a neighborhood of $t_0$ on which we have $\Psi(t) \neq \emptyset$ and 
$$
\mathrm{d}_H(\Psi(t_0), \Psi(t)) \leq 2 \sqrt{\frac{L}{\kappa} \mathrm{d}_\mathcal{T}(t,t_0)}.
$$
\end{lemma}

\begin{proof}
Conditions S1 and C1 imply that $\Psi(t_0)$ is a singleton $\lbrace x_0 \rbrace \subset V$. By C2 and the openess of $V$, there thus exist constants $L > 0$ and $r > 0$ such that \eqref{LipschitzCond} holds for all $t$ in a neighborhood of $t_0$ and all $x \in \mathrm{cl} \; B_r(x_0) \subset V$. In particular, setting $Q := B_r(x_0)$, assumptions S2 and S4 of Lemma \ref{StabilityKlatte} are fulfilled. By C1, $g_{t_0}(\cdot) := g(\cdot,t_0)$ is $\kappa$-strongly convex on $V$ and \cite[Proposition 4.2]{GoebelRockafellar2008} yields
\begin{equation}
\label{GoebelRockafellar}
g_{t_0}(x) \geq g_{t_0}(x_0) + d^\top (x-x_0) + \frac{1}{2} \kappa \| x-x_0 \|^2
\end{equation}
for all $d \in \partial g_{t_0}(x_0)$ and all $x \in V$. As $x_0$ minimizes $g(\cdot,t_0)$ over $X$, there is a subgradient $d_0 \in \partial g_{t_0}(x_0)$ such that $d_0^\top (x - x_0)$ is nonnegative for all $x \in X$. Consequently, \eqref{GoebelRockafellar} implies
$$
g_{t_0}(x) \geq g_{t_0}(x_0) + \frac{1}{2} \kappa \| x-x_0 \|^2
$$
for all $x \in X \cap Q$. Choosing $\alpha(\cdot) := \frac{1}{2} \kappa \| \cdot \|$ we obtain S3. Lemma \ref{StabilityKlatte} yields the existence of a neighborhood $T_0$ of $t_0$ on which $\Psi(t) \neq \emptyset$ and \eqref{KlatteEstimate} hold. Thus,  
$$
\mathrm{d}_H(\Psi(t_0), \Psi(t)) = \sup_{x \in \Psi(t)} \| x - x_0 \| \leq 2 \sqrt{\frac{L}{\kappa} \mathrm{d}_\mathcal{T}(t,t_0)}
$$
holds for all $t \in T_0$.
\end{proof}

Returning to stochastic programming models, we shall endow the parameter space of Borel probability measures on $\mathbb{R}^s$ with finite moments of order $p \geq 1$
$$
\mathcal{M}^p_s := \lbrace \mu \in \mathcal{P}(\mathbb{R}^s) \; | \; \int_{\mathbb{R}^s} \|t\|^p~\mu(\text{d}t) < \infty \rbrace
$$
with the $p-$th order Wasserstein distance (cf. \cite{RaRue98}, \cite{Vil03} and \cite{Vil09})
$$
W_p(\mu, \nu) := \inf_{\kappa} \Big\{ \int_{\mathbb{R}^s \times \mathbb{R}^s} \|v-\tilde{v}\|^p~\kappa(\text{d}(v,\tilde{v})) \; | \; \kappa \in \mathcal{P}(\mathbb{R}^s \times \mathbb{R}^s), \kappa \circ \pi_1^{-1} = \mu, \; \pi_2^{-1} = \nu \Big\}^{\frac{1}{p}}.
$$
To make the dependence of $Q_{\alpha \, CVaR}$ on the underlying measure explicit, we shall consider the mapping $Q_{\alpha CVaR}: \mathbb{R}^s \times \mathcal{M}^1_s \to \mathbb{R}$ definded by
$$
Q_{\alpha \, CVaR}(x, \mu) := \min_{\eta \in \mathbb{R}} \; \eta + \frac{1}{\alpha} \int_{\mathbb{R}^s} \max \lbrace 0, \varphi(z-x) - \eta\rbrace~\mu(\text{d}z).
$$
Let $g: \left[0,1\right] \rightarrow \left[0,1\right]$ be some continuous, increasing, concave distortion function with $g(0) = 0, g(1) = 1$ such that the mapping $Q_{\nu_g}: \mathbb{R}^n \times \mathcal{M}^p_s \to \mathbb{R}$,
$$
Q_{\nu_g}(x, \mu) = \int_{0}^{1} Q_{\alpha\,CVaR}(x,\mu) \, \nu(\text{d}\alpha)
$$
with $\nu_g \in \mathcal{P}((0,1])$ given by \eqref{Distortion1} and \eqref{Distortion2} is well defined. We shall consider the parametric optimisation problem
\begin{equation}
\leqnomode
\label{Pmu}
\tag*{$\textbf{P($\mu$)}$}
\min_x \lbrace \hat{h}(x) + Q_{\nu_g}(x,\mu) \; | \; x \in T(X)) \rbrace,
\end{equation}
where $X$ is some subset of $\mathbb{R}^n$, $T: \mathbb{R}^n \to \mathbb{R}^s$ is linear and the mapping $\hat{h}$ is given by \eqref{Hath}. Let $\Psi_g: \mathcal{M}^p_s \rightrightarrows \mathbb{R}^n$,
$$
\Psi_g(\mu) := \mathrm{Argmin}_x \lbrace \hat{h}(x) + Q_{\nu_g}(x,\mu) \; | \; x \in T(X)) \rbrace
$$
denote the optimal solution set mapping of \eqref{Pmu}.

\begin{theorem}[Quantitative Stability of \eqref{Pmu}]
Let $X \subseteq \mathbb{R}^n$ be nonempty, closed and convex and let $\mu_0 \in \mathcal{M}^p_s$ be such that $Q_{\nu_g}(\cdot, \mu_0)$ is $\kappa$-strongly convex on some nonempty, open convex set $V$ satisfying $\Psi_g(\mu_0) \cap V \neq \emptyset$. Furthermore, assume that
$$
\mathcal{L} := \int_{0}^{1} \frac{1}{(1-\alpha)^{\frac{1}{p}}}~\nu_g(\mathrm{d}\alpha) < \infty.
$$
is finite. Then there exists a constant $r > 0$ such that for any $\mu \in \mathcal{M}^p_s$ satisfying $d_p(\mu, \mu_0) \leq r$ we have $\Psi_g(\mu) \neq \emptyset$ and 
$$
\mathrm{d}_H(\Psi_g(\mu_0), \Psi_g(\mu)) \leq 2 \sqrt{\frac{\mathcal{L} \cdot \max_{i \in I} \|d_i \|}{\kappa} \cdot \mathrm{d}_p(\mu_0,\mu)}.
$$
\end{theorem}

\begin{proof}
By \cite[Corollary 12]{Pichler2013}, the first part of Lemma \ref{lemma0} and finiteness of $\mathcal{L}$ imply
$$
|Q_{\nu_g}(\mu) - Q_{\nu_g}(\mu')| \leq \mathcal{L} \cdot \max_{i \in I} \|d_i \| \cdot \mathrm{d}_p(\mu,\mu')
$$
for any $\mu, \mu' \in \mathcal{M}^p_s$. In addition, the linearity of $T$ implies that $T(X)$ is nonempty, closed and convex. The result is thus a direct consequence of Lemma \ref{LemmaStronglyConvexStability}.
\end{proof}

\section{Appendix}
\subsection*{Example 2: Details} 
\begin{align}
Q_{EE}(x_1,x_2) &= \int_0^1 \int_0^1 \max\{ 0, \max\{ 0,z_1-x_1,z_2-x_2 \} - \eta \} \, \text{d}z_1 \text{d}z_2 \nonumber \\
&= \int_{0}^{x_2} \int_{x_1}^{1} \max\{ 0, z_1 - x_1 - \eta \} \, \text{d}z_1 \text{d}z_2 \tag{A} \\
&+ \int_{0}^{x_1} \int_{x_2}^{1} \max\{ 0, z_2 - x_2 - \eta \} \, \text{d}z_2 \text{d}z_1 \tag{A'} \\
&+ \int_{x_2}^{1} \int_{x_1}^{1} \max\{ 0, \max\{z_1 - x_1, z_2 - x_2 \} - \eta \} \, \text{d}z_1 \text{d}z_2 \tag{B} 
\end{align}
We directly get
\begin{align*}
(A) &= \frac{1}{2}x_2(1-x_1-\eta)^2 \\
(A')&=  \frac{1}{2}x_1(1-x_2-\eta)^2.
\end{align*}
The calculation of (B) is a little more involved:
\begin{align}
(B) &=  \int_{x_2}^{1+x_2-x_1} \int_{x_1+z_2-x_2}^{1} \max\{ 0, z_1 - x_1 - \eta \} \, \text{d}z_1 \text{d}z_2 \tag{C} \\
&+ \int_{x_2}^{1+x_2-x_1} \int_{x_1}^{x_1+z_2-x_2} \max\{ 0, z_2 - x_2 - \eta \} \, \text{d}z_1 \text{d}z_2 \tag{C'} \\
&+ \int_{1+x_2-x_1}^{1} \int_{x_1}^{1} \max\{ 0, z_2 - x_2 - \eta \} \, \text{d}z_1 \text{d}z_2 \tag{C''}
\end{align}
For the treatment of $(C)$ use that $x_1+ \eta \leq 1$ which implies
$x_2+\eta \leq 1$:
\begin{align*}
(C) &=  \int_{x_2}^{1+x_2-x_1} \int_{x_1+z_2-x_2 \lor x_1+\eta}^{1} z_1 - x_1 - \eta \, \text{d}z_1 \text{d}z_2 \\
&= \int_{x_2}^{x_2+\eta} \int_{x_1+\eta}^{1} z_1 - x_1 - \eta \, \text{d}z_1 \text{d}z_2 \\
&+ \int_{x_2+\eta}^{1+x_2-x_1} \int_{x_1+z_2-x_2}^{1} z_1 - x_1 - \eta \, \text{d}z_1 \text{d}z_2 \\
&= \frac{1}{2} \, \eta \, (1-x_1-\eta)^2 + \frac{1}{2} \int_{x_2+\eta}^{1+x_2-x_1} (1-x_1-\eta)^2 - (z_2-x_2-\eta)^2 \,
\text{d}z_2 \\
&= \frac{1}{2} \, \eta \, (1-x_1-\eta)^2 + \frac{1}{3} (1-x_1-\eta)^3.
\end{align*}
\begin{align*}
(C') &=  \int_{x_2}^{1+x_2-x_1} (z_2-x_2) \max\{ 0, z_2 - x_2 - \eta \} \, \text{d}z_2 \\
&= \int_{x_2 + \eta}^{1+x_2-x_1} (z_2 - x_2)(z_2 - x_2 - \eta) \, \text{d}z_2 \\
&= \int_{0}^{1-x_1-\eta} (z_2'+\eta) z_2' \, \text{d}z_2' \\
&= \frac{1}{3}(1-x_1-\eta)^3 + \frac{1}{2} \eta (1-x_1-\eta)^2.
\end{align*}
We thus get
\begin{align*}
(C) + (C') = \frac{2}{3}(1-x_1-\eta)^3 + \eta (1-x_1-\eta)^2,
\end{align*}

\begin{align*}
(C'') &=  \int_{1+x_2-x_1}^{1} (1-x_1) \max\{ 0, z_2 - x_2 - \eta \} \, \text{d}z_2 \\
&= \int_{1+x_2-x_1}^{1} (1-x_1) (z_2 - x_2 - \eta) \, \text{d}z_2 \\
&= \frac{1}{2}(1-x_1) [ (1-x_2-\eta)^2 - (1-x_1-\eta)^2 ].
\end{align*}

\subsection*{Proof of Lemma \ref{MonMap}:}
\begin{proof}
We shall first show that \eqref{MonMap} implies
\begin{align}\label{MonMap2}
f(x_2,y_2) - f(x_1,y_1) \geq  c \| x_1 - x_2 \|^2 + f'(x_1,y_1)(x_2-x_1, y_2-y_1)
\end{align}
Set $t_k = \frac{k}{m+1}$ for $k=0,\ldots,m+1$ and some arbitrary integer $m$. 
Let $z_1 = (x_1,y_1)$ and $z_2 = (x_2,y_2)$.
By the mean-value theorem we get $t_k < s_k < t_{k+1}$ such that
\begin{align*}
f(z_1 + t_{k+1}(z_2-z_1)) - f(z_1 + t_{k}(z_2-z_1)) \\
= f'(z_1 + s_k (z_2 - z_1)) (t_{k+1}-t_k)(z_2 - z_1).
\end{align*}
This yields
\begin{align*}
f(z_2) - f(z_1) &= \sum_{k=0}^{m} [ f(z_1 + t_{k+1}(z_2-z_1)) - f(z_1 + t_{k}(z_2-z_1)) ] \\
                    &= \sum_{k=0}^{m} [ f'(z_1 + s_k(z_2 - z_1)) - f'(z_1) ] (t_{k+1} - t_k) (z_2 - z_1)
                    + f'(z_1)(z_2-z_1) \\
                    &\geq \kappa \| x_2 - x_1 \|^2 \sum_{k=0}^{m} (t_{k+1} - t_k)s_k + f'(z_1)(z_2 - z_1).
\end{align*}
Since we have
\begin{align*}
\sum_{k=0}^{m} (t_{k+1} - t_k) s_k \geq \sum_{k=0}^{m} (t_{k+1} - t_k) t_k = \frac{1}{(m+1)^2} \sum_{k=0}^{m} k
= \frac{m}{2(m+1)} \rightarrow \frac{1}{2}, \ \ m \rightarrow \infty
\end{align*}
this shows \eqref{MonMap2} which is the same as
\begin{align}\label{MonMap3}
f(z_2) - f(z_1) \geq f'(z_1)(z_2-z_1) + \frac{\kappa}{2} \|x_2-x_1\|^2.
\end{align}
Now let $z = \lambda z_1 + (1-\lambda) z_2$. \eqref{MonMap3} yields
\begin{align*}
&f(z_1) - f(z) \geq f'(z) (z_1 - z) + \frac{\kappa}{2} \|x_1 - [\lambda x_1 + (1-\lambda)x_2] \|^2 \\
&f(z_2) - f(z) \geq f'(z) (z_2 - z) + \frac{\kappa}{2} \|x_2 - [\lambda x_1 + (1-\lambda)x_2] \|^2.
\end{align*}
Multiplying by $\lambda$ and $(1-\lambda)$ respectively and adding up gives
\begin{align*}
&\lambda f(z-1) + (1-\lambda) f(z_2) - f(z) \\
\geq \ &f'(z)[ \lambda(z_1 - z) + (1-\lambda) ( z_2 - z ) ] + \frac{\kappa}{2} \big[ \lambda 
\| x_1 - [ \lambda x_1 + (1-\lambda) x_2 ] \|^2 + \\ &(1-\lambda) \| x_2 - [\lambda x_1 + (1\lambda) x_2] \|^2 \big] \\
= \ &\frac{\kappa}{2} \lambda (1-\lambda) \| x_1 - x_2 \|^2
\end{align*}
due to the first bracketed term vanishing. Rearranging terms yields \eqref{PSC}.
\end{proof}

\begin{example}
We compute the $\alpha CVaR$ for an elementary example via representation \eqref{varrep}.\\  
Consider $\varphi(t) = \max\{t,-t\}$, $\mu = \frac{1}{2}\mathbb{U}_{| (0,1)} + 
\frac{1}{2}\mathbb{U}_{| (1.5,2.5)}$, $x \in (0,1)$ and $\eta > 0$:\\
For the expected excess we get
\begin{align*}
4 \, Q_{EE}(x,\eta) = 
\begin{cases}
(x-\eta)^2 + (1-x-\eta)^2 + \\
+ (\frac{5}{2}-x-\eta)^2 -  (\frac{3}{2}-x-\eta)^2 & 0\leq x-\eta \leq x+\eta \leq 1, \\
(1-x-\eta)^2 + (\frac{5}{2}-x-\eta)^2 -\\
- (\frac{3}{2}-x-\eta)^2& x-\eta \leq 0 \leq x+\eta \leq 1, \\
(\frac{5}{2}-x-\eta)^2 -  (\frac{3}{2}-x-\eta)^2& x-\eta \leq 0 \leq 1 \leq x+\eta \leq \frac{3}{2}, \\
(\frac{5}{2}-x-\eta)^2 & x-\eta \leq 0 \leq \frac{3}{2} \leq x+\eta \leq \frac{5}{2}, \\
(x-\eta)^2 + (\frac{5}{2}-x-\eta)^2 -\\-(\frac{3}{2}-x-\eta)^2& 0 \leq x-\eta \leq 1 \leq x+\eta \leq \frac{3}{2}, \\
(x-\eta)^2 + (\frac{5}{2}-x-\eta)^2& 0 \leq x-\eta \leq \frac{3}{2} \leq x+\eta \leq \frac{5}{2}, \\
0& x-\eta \leq 0 \leq \frac{5}{2} \leq x+\eta. \\
\end{cases}
\end{align*}
In order to calculate $Q_{\alpha VaR}(x)$ we first compute
\begin{align*}
&  \mu( \varphi(z-x) \leq t) \\
&= \frac{1}{2} \mathbb{U}_{| \ (0,1)} ([x-t,x+t]) + \frac{1}{2} \mathbb{U}_{| \ (1.5,2.5)} ([x-t,x+t]) \\
&=\begin{cases}
t,& 0 \leq x-t \leq x+t \leq 1, \\
\frac{1}{2}(x+t),& x-t \leq 0 \leq x+t \leq 1, \\
\frac{1}{2} + \frac{1}{2}(x+t-\frac{3}{2}),& x-t \leq 0 \leq \frac{3}{2} \leq x+t \leq \frac{5}{2}, \\
\frac{1}{2}(1-x+t),& 0 \leq x-t \leq 1 \leq x+t \leq \frac{3}{2}, \\
\frac{1}{2}(1-x+t) + \frac{1}{2}(x+t-\frac{3}{2}) ,& 0 \leq x-t \leq \frac{3}{2} \leq x+t \leq \frac{5}{2}, \\
1,& x-t \leq 0 \leq \frac{5}{2} \leq x+t. \\
\end{cases}
\end{align*}
For given $x$ and $\alpha$ now determine the minimal $t^*$ for which $\mu(\varphi(z-x) \leq t) \geq 1-\alpha$,
i.e. $t^* = Q_{\alpha VaR}(x)$. For $0 < \alpha < \frac{1}{4}$ we get $Q_{\alpha VaR}(x) > 1$ on the entire set $V$, plugging this into \eqref{varrep} exhibits $Q_{\alpha CVaR}$ to be nowhere strongly convex on $V$. For values of $\alpha$ close to $1$ and $x$ 
close to $0$ one also sees $Q_{\alpha CVaR}$ failing to be strongly convex. This is 
due to the fact that $Q_{\alpha VaR}$ is decreasing in a neighborhood of $0$.  
For $\frac{1}{4} < \alpha < \frac{1}{2}$ one can show strong convexity on 
$U = (\frac{5}{4} - \alpha, 1)$. 
\end{example}

\end{document}